\documentclass{article}
\usepackage{graphicx}
\graphicspath{Fig/}

\usepackage[utf8]{inputenc}
\usepackage[margin=1in]{geometry}
\usepackage{color}
\usepackage{enumitem}
\usepackage{authblk}
\usepackage{lineno}

\usepackage{hyperref}
\usepackage{hypcap}

\usepackage[ruled,vlined]{algorithm2e}
\usepackage{algorithmic}

\usepackage{float}
\usepackage{caption}
\usepackage{subcaption}
\usepackage{setspace}
\usepackage{tcolorbox}

\usepackage{amsmath,amsfonts,amssymb,mathtools, amsthm}
\usepackage[english]{babel}
\newtheorem{prop}{Proposition}
\newtheorem{theorem}{Theorem}
\newtheorem{defi}{Definition}
\newtheorem{coro}{Corollary}
\newtheorem{lemma}{Lemma}
\newtheorem{remark}{Remark}
\numberwithin{equation}{section}

\title{A structure-preserving multiscale solver for particle-wave interaction in non-uniform magnetized plasmas}

\author[*]{Kun Huang}
\author[$\dag$]{Irene M. Gamba}
\author[$\ddag$]{Chi-Wang Shu}

\affil[*]{Institute for Fusion Studies, University of Texas at Austin}
\affil[$\dag$]{Oden Institute for Computational Sciences and Engineering, University of Texas at Austin}
\affil[$\ddag$]{Division of Applied Mathematics, Brown University}

\date{}

\begin{document}
\maketitle

\begin{abstract}
    Particle-wave interaction is of fundamental interest in plasma physics, especially in the study of runaway electrons in magnetic confinement fusion. Analogous to the concept of photons and phonons, wave packets in plasma can also be treated as quasi-particles, called plasmons. To model the ``mixture" of electrons and plasmons in plasma, a set of ``collisional" kinetic equations has been derived, based on weak turbulence limit and the Wentzel-Kramers-Brillouin (WKB) approximation.

    There are two main challenges in solving the electron-plasmon kinetic system numerically. Firstly, non-uniform plasma density and magnetic field results in high dimensionality and the presence of multiple time scales. Secondly, a physically reliable numerical solution requires a structure-preserving scheme that enforces the conservation of mass, momentum, and energy.

    In this paper, we propose a struture-preserving multiscale solver for particle-wave interaction in non-uniform magnetized plasmas. The solver combines a conservative local discontinuous Galerkin (LDG) scheme for the interaction part with a trajectory averaging method for the plasmon Hamiltonian flow part. Numerical examples for a non-uniform magnetized plasma in an infinitely long symmetric cylinder are presented. It is verified that the LDG scheme rigorously preserves all the conservation laws, and the trajectory averaging method significantly reduces the computational cost.
\end{abstract}

\section{Introduction}\label{intro}
The interaction between charged particles and waves is a critical aspect of plasma behavior, which occurs across multiple time scales. The wave frequency in plasmas is extremely high. If one directly simulates the system using the Vlasov–Maxwell equations, the time step of the numerical method must be smaller than the wave period. In order to simulate waves within reasonable computational cost, they are modeled as quasi-particles (plasmons) under WKB approximation\cite{breizman_physics_2019}, which have Hamilton's equations of motion in the phase space. At the same time, their interaction with particles are modeled by the quasilinear theory \cite{vedenov1961nonlinear,drummond1962non, landau1946vibrations}, which renders two additional collision-like integral-differential operators. It turns out that, such interaction results in diffusion of electrons and reaction of plasmons. For a comprehensive introduction of the model, we refer the readers to Chapter 23 of Thorne \& Blandford's book \cite{thorne2017modern}.

The two components of this model—the Hamiltonian flow of plasmons and the electron-plasmon interaction—give rise to two numerical challenges. 

The first challenge is on efficiency. In practice, we are primarily interested in how the electron distribution is affected by the electron–plasmon interaction, rather than the detailed phase-space distribution of the plasmons themselves. However, the Hamiltonian flow of plasmons involves the smallest time scale in the entire model. As a result, solving this flow directly would waste significant computational resources on information that is not of interest. Fortunately, the Hamiltonian of the plasmons is time-independent, which means their trajectories are fixed. This allows us to eliminate the Hamiltonian flow term from the equation through trajectory averaging.

The trajectory average technique \cite{bostan2010transport} has been successful in tackling problems with simple advection field lines. For example, the gyro-average of Vlasov equations, where the average operator is simply integration along circles. For more complicated problems, there is the ray-tracing method \cite{aleynikov2015stability}, where sample trajectories are calculated with the Hamilton’s equation, and numerical averaging is performed on these trajectories. However, such approach is expensive, not conservative, and not compatible with finite element or finite difference solvers. To address these issues, we conducted an in-depth study of the trajectory-averaging method in the language of functional analysis. It turns out that, the trajectory-averaging operator is essentially an orthogonal projection onto the null space of the advection (Hamiltonian flow) operator \cite{bostan2010transport}. Based on this perspective, we propose not to discretize the trajectory-averaging operator directly. Instead, we first discretize the null space of the advection operator and then construct the corresponding projection operator.

The second challenge is on accuracy. As a reduced model stemming of the Vlasov-Maxwell system, our kinetic system of electrons and plasmons inherits the conservation properties, which relies on the delicate structure of the ``collision" operators and the Poisson bracket associated to plasmon dynamics. The goal of this paper is to develop a numerical solver that preserves the structure of the underlying physical systems, particularly for highly magnetized plasmas with non-uniform density profiles. The solver is designed to respect the conservation of mass, momentum, and energy while capturing the kinetic evolution of both particles and waves (electrons and plasmons).

In 2023, Huang et. al. \cite{huang2023conservative} proposed a structure-preserving solver for quasilinear theory, in which a continuous Galerkin method was employed to discretize the electron diffusion operator. However, in practical applications \cite{breizman_physics_2019}, the electron kinetic equation involves many additional terms, notably convection terms. To ensure compatibility with such more general models, we extend the idea in \cite{huang2023conservative} and propose a novel discontinuous Galerkin version.

This paper is organized as follows. In section \ref{sec:formulation} we describe the set up of the problem. After that we derive the weak form of trajectory-averaged equation in section \ref{sec:aver} and the connection-proportion algorithm in section \ref{sec:trajbundle}. The conservative LDG method for quasilinear theory will be introduced in section \ref{sec:ldg}. Implementation and complexity analysis will be discussed in section \ref{sec:implement}. Section \ref{sec:cylinder_results} contains the numerical results.

\section{The kinetic system for particles and plasmons} \label{sec:formulation}

Consider an infinitely long cylinder $\Omega_{x}=(0,R)\times(0,2\pi)\times\mathbb{R}$. At any given location $(r,\phi,z)\in\Omega_{x}$, there is a set of local
orthonormal basis $(\mathbf{e}_{z},\mathbf{e}_{r},\mathbf{e}_{\phi})$.

Suppose a plasma is confined in the cylinder, embedded in a magnetic field along the z-axis, $\mathbf{B}(\mathbf{x})=B(r)\mathbf{e}_{z}$.

We focus on the $z,\phi$-symmetric case for simplicity, which means any function $g(\mathbf{x})$ depends only on $r$.

The momentum $\mathbf{p}$ of a particle at $(r,\phi,z)$ can be decomposed as $\mathbf{p}=p_{z}\mathbf{e}_{z}+p_{r}\mathbf{e_{r}}+p_{\phi}\mathbf{e}_{\phi} \in \mathbb{R}_{p}^{3}$.  Define $p_{\parallel}=\frac{\mathbf{B}}{|\mathbf{B}|}\cdot\mathbf{p}$ and $p_{\perp}=\left\vert\left(I-\frac{\mathbf{B}}{|\mathbf{B}|}\otimes\frac{\mathbf{B}}{|\mathbf{B}|}\right)\cdot\mathbf{p}\right\vert$. Then for highly magnetized plasma, the particle probability density function $f(\mathbf{p},\mathbf{x},t)=f(p_{\parallel},p_{\perp},r,t)$.

The momentum  $\mathbf{k}$ of a plasmon (the wave vector of a wave packet) can also be decomposed as $\mathbf{k}=k_{z}\mathbf{e}_{z}+k_{r}\mathbf{e_{r}}+k_{\phi}\mathbf{e}_{\phi} \in  \mathbb{R}_{k}^{3}$. Define $q_{\phi} = k_{\phi}\cdot r$, then $(k_{r},q_{\phi},k_{z})$ and $(r, \phi, z)$ are a set of canonical coordinates. Given $z,\phi$-symmetry, the plasmon probability density function $N(\mathbf{k}, \mathbf{x}, t) = N(k_{r}, q_{\phi}, k_{z}, r, t)$.

\bigskip
The following system governs the evolution of particle \textit{pdf} $f(\mathbf{p}, \mathbf{x}, t)$ and plasmon \textit{pdf} $N(\mathbf{k}, \mathbf{x}, t)$,
\begin{equation} \label{kineticsytem}
\begin{cases}
    \partial_{t}f +v_{z}\partial_{z}f&=\nabla_{p}\cdot(D[N]\cdot\nabla_{p}f),\\
    \partial_{t}N +\left\{\omega,N\right\} &=\Gamma[f]N,
\end{cases}
\end{equation}
where the quasilinear diffusion term and the reaction term accounts for electron-plasmon interaction \cite{breizman_physics_2019}. The advection term $v_{z}\partial_{z}f$ comes from the gyro-averaged Vlasov equation, it vanishes since $f$ does not depend on $z$. In addition, the Poisson bracket $\left\{ \omega,N\right\} =\frac{\partial N}{\partial r}\frac{\partial\omega}{\partial k_{r}} - \frac{\partial\omega}{\partial r}\frac{\partial N}{\partial k_{r}}$ is a result of WKB approximation. Denote it as $\mathcal{T}N$, then obviously the operator $\mathcal{T}$ is anti-symmetric. 

The transport of plasmons is fully determined by the Hamiltonian $\omega(\mathbf{k}, \mathbf{x})$, i.e. the dispersion relation of waves in plasma. For a detailed discussion of wave dispersion relation in homogeneous plasmas, we refer the readers to section 3.2 of \cite{huang2023numerical}. As have been shown in \cite{huang2023numerical}, given the plasma frequency $\omega_{pe}$ and electron gyro-frequency $\omega_{ce}$, wave frequency $\omega$ is a function of wave-vector $\mathbf{k}$. Note that the plasma frequency $\omega_{pe}$ is proportional to the square root of electron number density $n_{e}$ and the gyro-frequency $\omega_{ce}$ is proportional to the
background magnetic field $B$. In this paper we assume that both $B$ and $n_{e}$ are dependent on the radial variable $r$. Therefore the dispersion relation should be $\omega = \omega(\mathbf{k}; \omega_{pe}(\mathbf{x}), \omega_{ce}(\mathbf{x})) = \omega(\mathbf{k}; \mathbf{x})$.

Test the equations with $\varphi(\mathbf{p},\mathbf{x})$ and $\eta(\mathbf{k},\mathbf{x})$ to obtain the following weak form,
\begin{equation}
    \iint_{px}\varphi\partial_{t}f+\iint_{kx}\eta\partial_{t}N=\iint_{kx}N\mathcal{T}\eta+\iiint_{pkx}\left(\eta\mathcal{L}E-\omega\mathcal{L}\varphi\right)\mathcal{L}fN\mathcal{B},
\end{equation}

where $E(\mathbf{p})$ is the kinetic energy of a single particle,

\begin{equation*}
    E(\mathbf{p}) = \gamma(\mathbf{p})mc^2 = \sqrt{m^2c^4 + p^2c^2},
\end{equation*}

the directional differential operator $\mathcal{L}$ is defined as,

\begin{equation}\label{operatorL}
    \mathcal{L}g\coloneqq k_{\parallel}\frac{\partial E}{\partial p_{\perp}}\frac{p}{p_{\perp}}\frac{\partial g}{\partial p_{\parallel}}+\left(\omega-k_{\parallel}\frac{\partial E}{\partial p_{\parallel}}\right)\frac{p}{p_{\perp}}\frac{\partial g}{\partial p_{\perp}},
\end{equation}

and the emission/absorption kernel $\mathcal{B}$ reads,
\begin{equation}\label{kernel}
    \mathcal{B}(\mathbf{p},\mathbf{k},\mathbf{x})\coloneqq\sum_{l}\frac{1}{\omega^{2}(\mathbf{k},\mathbf{x})}U_{l}(\mathbf{p},\mathbf{k},\mathbf{x})\delta(\omega(\mathbf{k},\mathbf{x})-k_{\parallel}v_{\parallel}-l\omega_{c}(\mathbf{x})/\gamma(\mathbf{p})).
\end{equation}

\section{Trajectorial average}\label{sec:aver}
As can be observed in Equation(\ref{kineticsytem}), there are three different time scales associated with the model, $\tau_\text{diff}$, $\tau_\text{adv}$ and $\tau_\text{reac}$. According to Kiramov and Breizman \cite{kiramov2021reduced}, the advection process is much faster than the other two, $\tau_\text{adv} \ll \tau_\text{diff} \sim \tau_\text{reac}$. In practice, particle-wave interaction is of interest, rather than plasmon advection. Hence it is reasonable to eliminate the Poisson bracket term through trajectorial averaging. 


\subsection{The Liouville-reaction equation}

Consider the Liouville-reaction equation 
\begin{equation}\label{LiouvilleReaction}
    \frac{\partial N(k,x,t)}{\partial t}+\frac{1}{\varepsilon}\left\{ N(k,x,t),H(k,x)\right\} =\Gamma(k,x,t)N(k,x,t),
\end{equation}
on a bounded connected domain $\Omega_{b}$, with $H(k,x) \equiv H_{b}$ and $\nabla H \neq 0 $ on $\partial \Omega_{b}$. Then the problem is well-posed without any boundary condition.



We use the Hilbert expansion to obtain an averaged equation \cite{bostan2010transport}.

To begin with, define the advection operator $\mathcal{T}$ as follows,

\begin{equation}
    \mathcal{T}g\coloneqq\nabla_{k}H\cdot\nabla_{x}g-\nabla_{x}H\cdot\nabla_{k}g.
\end{equation}

Then Equation(\ref{LiouvilleReaction}) can be written as,
\begin{equation}\label{adv-reaction}
    \partial_{t}N+\frac{1}{\varepsilon}\mathcal{T}N=\Gamma N.
\end{equation}

In addition, assume that $N(k,x,t)$ has the following Hilbert expansion,
\begin{equation}\label{HilbertExpansion}
    N=N_{0}+\varepsilon N_{1}+\varepsilon^{2}N_{2}+\cdots.
\end{equation}

Substitute Equation(\ref{HilbertExpansion}) into Equation(\ref{adv-reaction}) and collect the terms according to the power of $\varepsilon$.

The $\varepsilon^{-1}$ term being zero yields that,
\begin{equation}\label{ep-1}
    \mathcal{T}N_{0}=0.
\end{equation}

The $\varepsilon^{0}$ term being zero yields that,
\begin{equation}\label{ep0}
    \partial_{t}N_{0}+\mathcal{T}N_{1}=\Gamma N_{0}.
\end{equation}

Denote the orthogonal projection from $L^{2}(\Omega_{b})$ onto $L^{2}(\Omega_{b}) \cap \text{ker} \mathcal{T}$ as $\mathcal{P}$, then $\mathcal{P}u \in L^{2}(\Omega_{b}) \cap \text{ker}\mathcal{T}$, and $\left(\mathcal{P}u,\eta\right)_{\Omega_{b}}=\left(u,\eta\right)_{\Omega_{b}}$ for any $\eta \in L^{2}(\Omega_{b}) \cap \text{ker} \mathcal{T}$.

Apply the projection $\mathcal{P}$ on both sides of Equation(\ref{ep0}) to obtain
\begin{equation*}
    \partial_{t}\mathcal{P}N_{0}+\mathcal{P}\mathcal{T}N_{1}=\mathcal{P}\left(\Gamma N_{0}\right).
\end{equation*}

Equation(\ref{ep-1}) renders that $\mathcal{P}N_{0} = N_{0}$. In addition, note that the advection operator $\mathcal{T}$ is skew-symmetric, therefore $\mathcal{P}\mathcal{T}N_{1} = 0$. It follows that
\begin{equation}\label{AverEq}
    \partial_{t}N_{0}=\mathcal{P}\left(\Gamma N_{0}\right).
\end{equation}

We call $\mathcal{P}$ the trajectorial averaging operator, and Equation(\ref{AverEq}) is the averaged equation.

Since the projection operator $\mathcal{P}$ is self-adjoint, for $\eta \in L^{2}(\Omega_{b}) \cap \text{ker} \mathcal{T}$, we have
\begin{equation*}
    \left(\mathcal{P}\left(\Gamma N_{0}\right),\eta\right)_{\Omega_{b}}=\left(\Gamma N_{0},\mathcal{P}\eta\right)_{\Omega_{b}}=\left(\Gamma N_{0},\eta\right)_{\Omega_{b}}.
\end{equation*}

Therefore the weak form of the averaged problem can be stated as follows:

Find an $N_{0} \in  L^{2}(\Omega_{b}) \cap \text{ker} \mathcal{T}$ such that 
\begin{equation}\label{weakAverEq}
    \left(\partial_{t}N_{0},\eta\right)_{\Omega_{b}}=\left(\Gamma N_{0},\eta\right)_{\Omega_{b}},
\end{equation}
for any $\eta \in L^{2}(\Omega_{b}) \cap \text{ker} \mathcal{T}$.

\subsection{The averaged kinetic system}  

Since $\tau_\text{adv} \ll \tau_\text{diff}$, the diffusion equation for particle \textit{pdf} in (\ref{kineticsytem}) can be approximated with
\begin{equation*}
    \partial_{t}f=\nabla_{p}\cdot(D[N_{0}]\cdot\nabla_{p}f).
\end{equation*}

Testing the averaged system with $\varphi(\mathbf{p}, \mathbf{x}) \in \mathcal{G} = H^{1}(\mathbb{R}^{3}_{p})\otimes L^2(\Omega_{x})$ and $\eta(\mathbf{k}, \mathbf{x}) \in \mathcal{N}= \text{ker}\mathcal{T} \cap L^{2}(\mathbb{R}^{3}_{k}\times\Omega_{x})$ yields
\begin{equation}\label{aversys}
    \iint_{px}\varphi\partial_{t}f+\iint_{kx}\eta\partial_{t}N=\iiint_{pkx}\left(\eta\mathcal{L}E-\omega\mathcal{L}\varphi\right)\mathcal{L}fN\mathcal{B},
\end{equation}
where we have used $N$ instead of $N_{0}$ since the higher order terms in the Hilbert expansion (\ref{HilbertExpansion}) will be ignored in the rest of the paper.

\bigskip
Since the absorption/emission kernel $\mathcal{B}(\mathbf{p}, \mathbf{k}, \mathbf{x})$ as defined in Equation(\ref{kernel}) contains Dirac delta, for the purpose of either modeling or numerical implementation, it is
usually replaced with its approximation $\mathcal{B}_{\varepsilon}(\mathbf{p},\mathbf{k}, \mathbf{x})$. In our previous paper \cite{huang2023conservative}, it has been proved that our choice of directional differential operator $\mathcal{L}$ in Equation(\ref{operatorL}) guarantees unconditional conservation. In the following theorem we show that the unconditional conservation property is preserved even after trajectorial average.

    \begin{theorem}[unconditional conservation] \label{conserv}

If $f(\mathbf{p}, \mathbf{x}, t)$ and $N(\mathbf{k}, \mathbf{x}, t)$ solve Equation(\ref{aversys}) with emission/absorption kernel being replaced by $\mathcal{B}_{\varepsilon}$, then for any $\mathcal{B}_{\varepsilon}$ we have the following conservation laws,

\begin{itemize}
    \item Mass Conservation
    \begin{equation*}
        \partial_{t}\mathcal{M}_{tot}=\partial_{t}\left((f,1)_{px}+(N,0)_{kx}\right)=0,
    \end{equation*}
    
    \item Momentum Conservation
    \begin{equation*}
        \partial_{t}\mathcal{P}_{z,tot}=\partial_{t}\left((f,p_{z})_{px}+(N,k_{z})_{kx}\right)=0,
    \end{equation*}
    \item Energy Conservation
    \begin{equation*}
        \partial_{t}\mathcal{E}_{tot}=\partial_{t}\left((f,E)_{px}+(N,\omega)_{kx}\right)=0.
    \end{equation*}
    
\end{itemize}
\end{theorem}

\begin{proof}
We consider the conservation laws one by one. 

Mass conservation is trivial. The energy conservation law is also trivial, since $\omega \in \text{ker}\mathcal{T}$, and
\begin{equation*}
    \omega\mathcal{L}E-\omega\mathcal{L}E=0.
\end{equation*}

For momentum conservation, note that $k_{z}\in \text{ker}\mathcal{T}$, hence we have,
\begin{equation*}
\partial_{t}\mathcal{P}_{z,tot}=\iiint_{pkx}\left(k_{z}\mathcal{L}E-\omega\mathcal{L}p_{z}\right)\mathcal{L}fN\mathcal{B}_{\varepsilon}
\end{equation*}
Recall the definition
\begin{equation*}
\mathcal{L}g\coloneqq k_{\parallel}\frac{\partial E}{\partial p_{\perp}}\frac{p}{p_{\perp}}\frac{\partial g}{\partial p_{\parallel}}+\left(\omega-k_{\parallel}\frac{\partial E}{\partial p_{\parallel}}\right)\frac{p}{p_{\perp}}\frac{\partial g}{\partial p_{\perp}},
\end{equation*}
since $\mathbf{B} \parallel \mathbf{e}_{z}$, 
\begin{equation*}
    k_{z}\mathcal{L}E-\omega\mathcal{L}p_{z}=k_{z}\omega\frac{p}{p_{\perp}}\frac{\partial E}{\partial p_{\perp}}-\omega k_{\parallel}\frac{\partial E}{\partial p_{\perp}}\frac{p}{p_{\perp}}=0.
\end{equation*}
It follows that $\partial_{t}\mathcal{P}_{z,tot}=0$ regardless of $\mathcal{B}_{\varepsilon}$.
    
\end{proof}

\section{Trajectory bundle: definition and construction}\label{sec:trajbundle}

To discretize Equation(\ref{aversys}), it is necessary to construct a series of finite dimensional spaces approximating the test space $\text{ker}\mathcal{T} \cap L^{2}(\mathbb{R}^{3}_{k}\times\Omega_{x})$. Hence in this section we introduce the concept of trajectory bundle, and propose the connnection-proportion algorithm to construct basis functions.

\subsection{Trajectory bundles and their properties}

Consider a particle moving in phase space, denote the trajectory generated by initial state $(\mathbf{k}_{0},\mathbf{x}_{0})$ as $s(\mathbf{k}_{0},\mathbf{x}_{0})$, then the Hamiltonian $H(\mathbf{k}, \mathbf{x})$ must be a constant along the trajectory,
\begin{equation*}
    H(\mathbf{k},\mathbf{x})=H(\mathbf{k}_{0},\mathbf{x}_{0}),\forall(\mathbf{k},\mathbf{x})\in S(\mathbf{k}_{0},\mathbf{x}_{0}).
\end{equation*}

Define the trajectory bundle generated by set $\mathcal{A}$ as $S(\mathcal{A})=\cup_{(\mathbf{k}_{0},\mathbf{x}_{0})\in\mathcal{A}}S(\mathbf{k}_{0},\mathbf{x}_{0})$. The set $S(\mathcal{A})$ is the union of all the trajectories generated by initial states inside $\mathcal{A}$, therefore they share the same range of Hamiltonian,
\begin{equation*}
    H(S(\mathcal{A}))=H(\mathcal{A}).
\end{equation*}

We present the above definition in order to give the readers some physics intuition. The following approach is more convenient for implementation.

\begin{defi}{Trajectory bundles}

    Given $H_{a}$ and $H_{b}$ such that $\inf_{\Omega_{b}}H(k,x)\leq H_{a}<H_{b}\leq\sup_{\Omega_{b}}H(k,x)$. There exists a family of subsets of $\Omega_{b}$ with finite cardinality such that
    \begin{enumerate}
        \item The union of these subsets are the inverse image of the interval $\left(H_{a},H_{b}\right)$,
        \begin{equation}
            \cup_{S\in\mathcal{F}}S=\left\{ \left(k,x\right):H(k,x)\in\left(H_{a},H_{b}\right)\right\}.
        \end{equation}
        \item Any subset $S\in \mathcal{F}$ is connected.
        \item If $S_{1} \in \mathcal{F}$, $S_{2} \in \mathcal{F}$ and $S_{1} \neq S_{2}$, then $\overline{S_{1}} \cap \overline{S_{2}}$ is either empty or a countable set of points in $\Omega_{b}$.
    \end{enumerate}

    Each subset $S \in \mathcal{F}$ is called a trajectory bundle. And $\mathcal{F}$ is called the family of trajectory bundles generated by the interval $\left(H_{a},H_{b}\right)$.

\end{defi}

The basis function induced by a trajectory bundle $S$ is defined as 
    \begin{equation}
        \mathbf{1}_{S}=\begin{cases}
        1, & (k,x)\in S\\
        0, & \text{otherwise}
        \end{cases}
    \end{equation} 

Obviously $\{\mathbf{1}_{S}, H\} = 0$, i.e. $\mathbf{1}_{S} \in \text{ker}\mathcal{T} \cap L^{2}(\Omega_{b})$.

\bigskip
Given a segmentation $\sup_{\Omega_{b}}H(k,x)=H_{0}\leq H_{1}<H_{2}<\cdots<H_{N-1}\leq H_{N}=\inf_{\Omega_{b}}H(k,x)$, each interval $\left(H_{j},H_{j+1}\right)$ generates a family of trajectory bundles $\mathcal{F}_{j}$. Denote the basis functions induced by all these trajectory bundles as $\eta_{i}$, $i=1, \cdots, M$. They have the following properties,

\begin{enumerate}
    \item Partition of unity,
    \begin{equation}
    \cup_{i=1}^{M}\eta_{i}(k,x)=1,\ \forall (k,x) \in \Omega_{b}.
    \end{equation}
    
    \item Orthogonality
    \begin{equation}
        \left(\eta_{i},\eta_{j}\right)_{\Omega_{b}}=\delta_{ij}\left(\eta_{i},\eta_{i}\right)_{\Omega_{b}},
    \end{equation}
    where $\delta_{ij}$ is the Kronecker delta.

\end{enumerate}

Denote $\text{span}\{\eta_{i}\}$ as $X_{h}$. We seek for $N_{h} \in X_{h}$ such that
\begin{equation}\label{discrete_reaction}
    \left(\partial_{t}N_{h},\eta_{h}\right)=\left(\gamma N_{h},\eta_{h}\right),\forall\eta_{h}\in X_{h}.
\end{equation}


\subsection{The connection-proportion algorithm}
In what follows, we assume that the domain $\Omega_b \subset \mathbb{R}^{2} $ is partitioned into triangular meshes $ \{V_{i} \}$, and the Hamiltonian $H(k,x) \in C(\Omega_{b})$ is piecewise linear on the mesh. Under such assumption, the trajectory bundles are areas between polygons - the level sets of $H$. The challenge here is to design a data structure which stores all the necessary information associated to the trajectory bundles.

We propose the connection-proportion algorithm to construct the family $\mathcal{F}$ of trajectory bundles generated by interval $I = (H_{a}, H_{b})$.  The key is to distinguish different connected components of the inverse image $H^{-1}(I)$. To avoid ambiguity caused by saddle points, we assume that $H_{a}$ and $H_{b}$ are not equal to any node value of the piecewise linear Hamiltonian. 

Firstly, let us introduce some preliminary concepts.

\begin{defi}[minimal triangle cover]
    For $S\in \mathcal{F}$, define
    \begin{equation*}
        \mathcal{R}_{min}(S) \coloneqq \left\{ V_{i}:V_{i}\cap S\neq\emptyset\right\}.
    \end{equation*}
    It can be easily verified that for any triangle cover $\mathcal{R}$ such that 
    \begin{equation*}
        S\subset\cup_{V_{i}\in\mathcal{R}}\overline{V_{i}},
    \end{equation*}
    we have $\mathcal{R}_{min} \subset \mathcal{R}$. Hence we call $\mathcal{R}_{min}(S)$ the minimal triangle cover of $S$.
\end{defi}

The following proposition shows that the minimal triangle cover replicates the topological relation between the corresponding trajectory bundles. 
\begin{prop} 
    If $S_{i}, S_{j}\in \mathcal{F}$ and $S_{i} \neq S_{j}$, then 
    \begin{equation*}
        \mathcal{R}_{min}(S_{i}) \cap \mathcal{R}_{min}(S_{j}) = \emptyset.
    \end{equation*}
\end{prop}

\begin{coro}\label{distinguish}
    If $(k,x) \in \cup_{V_{i}\in\mathcal{R}_{min}(S)}\overline{V_{i}}$ and $H(k,x) \in I$, then
    \begin{equation*}
        (k,x) \in S.
    \end{equation*}
\end{coro}

Corollary \ref{distinguish} actually points to the proper way of implementation: it is neither necessary nor sufficient to store the nodes of polygons, all we need is the minimal triangle cover. In other words, we have transformed a continuous problem in topology into a discrete problem in graph theory. Next, we propose the algorithm to construct the minimal triangle covers.

\begin{defi}[connection matrix]
    For a given interval $I = (H_{a}, H_{b})$, the connection matrix associated with $I$ is defined as 
    \begin{equation}\label{connectionmat}
        A_{ij}(I)=\begin{cases}
1, & \overline{V_{i}}\cap\overline{V_{j}}\cap H^{-1}(I)\neq\emptyset\\
0, & \text{otherwise}
\end{cases}
    \end{equation}
\end{defi}

The connection matrix renders a finite family of connected components $\{\mathcal{R}_{\nu} \}$. The following theorem shows that each of them corresponds to a trajectory bundle.

\begin{theorem}
    Suppose that $\mathcal{F}(I)$ is the family of trajectory bundles generated by interval $I=(H_{a}, H_{b})$ and $\{\mathcal{R}_{\nu} \}$ is the set of connected components determined by the connection matrix $A_{ij}(I)$ as defined in Equation(\ref{connectionmat}), then
    \begin{equation*}
        \{\mathcal{R}_{\nu}\}=\left\{ \mathcal{R}_{min}(S):S\in\mathcal{F}(I)\right\}. 
    \end{equation*}
\end{theorem}

\begin{proof}
    It takes two steps to prove the theorem.
    \begin{enumerate}
        \item Prove that $\{\mathcal{R}_{\nu}\} \subset \left\{ \mathcal{R}_{min}(S):S\in\mathcal{F}(I)\right\} $.

        By definition, any connected component $\mathcal{R}_{\nu}$ contains at least two elements, say $V_{m}$ and $V_{n}$. 
        
        Since $H^{-1}(I)$ is open,
        \begin{equation*}
            \overline{V_{m}}\cap\overline{V_{n}}\cap H^{-1}(I)\neq\emptyset\Rightarrow V_{m}\cap H^{-1}(I)\neq\emptyset.
        \end{equation*}

        Considering that
        \begin{equation*}
            \emptyset \neq V_{m}\cap H^{-1}(I)=V_{m}\cap\left(\cup_{\mathcal{F}}S\right)=\cup_{\mathcal{F}}\left(V_{m}\cap S\right),
        \end{equation*}
        
        there must exists a trajectory bundle $S\in \mathcal{F}$ such that
        \begin{equation*}
            V_{m}\cap S \neq \emptyset.
        \end{equation*}

        It remains to show that $\mathcal{R}_{\nu}=\mathcal{R}_{min}(S)$.
        \begin{itemize}
            \item Prove that $\mathcal{R}_{\nu} \subset \mathcal{R}_{min}(S)$.

            Recall that $V_{m}\cap S \neq \emptyset$, then for any $V_{l}\in \mathcal{R}_{\nu}$, there must be a continuous path inside $H^{-1}(I)$ from $(k_{1}, x_{1}) \in V_{l}\cap H^{-1}(I)$ to  $(k_{2}, x_{2}) \in V_{m}\cap H^{-1}(I)$. 
            Since $(k_{2}, x_{2}) \in S$ implies that $(k_{1}, x_{1}) \in S$, it follows that $V_{l}\cap S \neq \emptyset$. Therefore $\mathcal{R}_{\nu} \subset \mathcal{R}_{min}(S)$.
            
            \item Prove that $\mathcal{R}_{\nu} \supset \mathcal{R}_{min}(S)$.

            Consider a triangle $V_{l} \in \mathcal{R}_{min}(S)$, then by definition $V_{l} \cap S \neq \emptyset$. Suppose that $V_{l} \notin \mathcal{R}_{\nu}$, then $(k_{1}, x_{1})\in V_{l} \cap S$ is not connected to any point $(k_{2}, x_{2}) \in S\cap\cup_{i\in\mathcal{R}_{\nu}}\overline{V_{i}}$, hence $S$ is not a connected set, contradictory to its definition, therefore $V_{n} \in \mathcal{R}_{\nu}$.
        \end{itemize}
        
        \item Prove that $\{\mathcal{R}_{\nu}\} \supset \left\{ \mathcal{R}_{min}(S):S\in\mathcal{F}(I)\right\} $.

        By definition $\cup_{\nu} \mathcal{R}_{\nu} = \cup_{\mathcal{F}} \mathcal{R}_{min}(S)$, hence for any $S$ there exists an $\mathcal{R}_{\nu}$ such that $\mathcal{R}_{\nu} \cap \mathcal{R}_{min}(S) \neq \emptyset$. In the same way as above, it can be proved that $\mathcal{R}_{\nu}=\mathcal{R}_{min}(S)$.
    \end{enumerate}
\end{proof}

In principle, once the minimal triangle covers are constructed, any integral can be performed by quadrature. However, recall the semi-discrete weak form 
\begin{equation*}
\left(\partial_{t}N_{h},\eta_{h}\right)_{\Omega_{b}}=\left(\gamma N_{h},\eta_{h}\right)_{\Omega_{b}}.
\end{equation*}

The mass matrix $\mathcal{M}_{ij} = \left(\eta_{i}, \eta_{j}\right)_{\Omega_{b}}$ plays the key role here, and we expect to calculate the precise values. This is possible thanks to the fact that 
\begin{equation*}
    \left(\eta_{i},\eta_{j}\right)_{\Omega_{b}}=\delta_{ij}\left(\eta_{i},\eta_{i}\right)_{\Omega_{b}} = \delta_{ij}\left(\eta_{i},1\right)_{\Omega_{b}} = \delta_{ij} \cdot \mu \left(S_{i}\right),
\end{equation*}
where $S_{i} = \text{supp}\left(\eta_{i}\right)$ is a trajectory bundle and $\mu(S_{i})$ represents its measure.

\begin{theorem}
    For a triangle $V$ with vertices $(k_{i}, x_{i})$, $i=1,2,3$. The proportion $\frac{\mu(V\cap H^{-1}(I))}{\mu(V)}$ only depends on the interval $(H_{a}$, $H_{b})$ and the node values $H_{i}= H(k_{i}, x_{i})$, $i=1,2,3$.
\end{theorem}

\begin{proof}
    Without loss of generality, we assume that $H_{1} = H_{2} < H_{a} < H_{3}$. Since $H$ is piecewise linear on the mesh, the set $V\cap H^{-1}(H_{a}, +\infty)$ is a similar triangle of $V$ with vertices 
    \begin{equation*}
        \begin{split}
            \left(\tilde{k_{1}},\tilde{x_{1}}\right)
            &=\frac{H_{3}-H_{a}}{H_{3}-H_{1}}\left(k_{1},x_{1}\right)+\frac{H_{a}-H_{1}}{H_{3}-H_{1}}\left(k_{3},x_{3}\right),\\
            \left(\tilde{k_{2}},\tilde{x_{2}}\right)
            &=\frac{H_{3}-H_{a}}{H_{3}-H_{1}}\left(k_{2},x_{2}\right)+\frac{H_{a}-H_{1}}{H_{3}-H_{1}}\left(k_{3},x_{3}\right),\\
            \left(\tilde{k_{3}},\tilde{x_{3}}\right)
            &=\left(k_{3},x_{3}\right).
        \end{split}
    \end{equation*}

    Therefore the following proportion is only dependent on $H_{1}$, $H_{3}$ and $H_{a}$.
    \begin{equation*}
        \frac{\mu(V\cap H^{-1}(H_{a},+\infty))}{\mu(V)}=\left(\frac{H_{3}-H_{a}}{H_{3}-H_{1}}\right)^{2}.
    \end{equation*}

    For the case where $H_{1} < H_{2} < H_{3}$, it is always possible to split the triangle into two by the level set w.r.t. $H_{2}$ and repeat the procedure above.
\end{proof}

Define the proportion vector as
\begin{equation*}
    r_{m}\coloneqq\frac{\mu(V_{m}\cap H^{-1}(I))}{\mu(V_{m})}.
\end{equation*}

Then for any given trajectory bundle $S$,  
\begin{equation*}
    \mu(S)=\sum_{V_{m}\in\mathcal{R}_{min}(S)}r_{m}\cdot\mu(V_{m}).
\end{equation*}

\begin{remark}
    The connection-proportion algorithm can be generalized to higher dimension straightforwardly, where one need to construct the minimal simplex cover instead.
\end{remark}

\section{The conservative scheme based on LDG method} \label{sec:ldg}
As have been proved in Theorem \ref{conserv}, the averaged system admits three conservation laws. In this section, we seek for a discretization that preserves all of them rigorously. The idea here is analogous to the one introduced in our previous work \cite{huang2023conservative}: to replace some quantities with their projection in the test spaces. In what follows, we generalize the idea from continuous Galerkin method to local discontinuous Galerkin method.

\subsection{The bilinear form of local DG method}

Consider the following diffusion equation in bounded domain $\Omega$,
\begin{equation}
    \partial_{t}f=\nabla\cdot\left(D\cdot\nabla f\right),
\end{equation}
with Neumann's boundary condition on $\partial \Omega$. 
\begin{equation*}
    D\cdot\nabla f = \mathbf{0}.
\end{equation*}

Suppose the domain $\Omega$ is partitioned into elements $\{R\}$.  The standard local discontinuous Galerkin method uses the following weak form,
\begin{equation} \label{weakform-ldg}
    \begin{split}
        \sum_{R}\left(\partial_{t}f,\varphi\right)_{R}
        &=-\sum_{R}\left(Z,\nabla\varphi\right)_{R}+\sum_{R}\langle\widehat{Z},\varphi^{-}\mathbf{n}^{-}\rangle_{\partial R/\Gamma}\\
        \sum_{R}\left(Z,\widetilde{V}\right)_{R}
        &=\sum_{R}\left(D\cdot\widetilde{Z},\widetilde{V}\right)_{R}\\
        \sum_{R}\left(\widetilde{Z},V\right)_{R}
        &=-\sum_{R}\left(f,\nabla\cdot V\right)_{R}+\sum_{R}\langle\widehat{f},V^{-}\cdot\mathbf{n}^{-}\rangle_{\partial R}
    \end{split}
\end{equation}

Choosing an appropriate reference vector $\mathbf{u}$, we use alternating flux as follows:
\begin{equation}\label{Zflux}
    \widehat{Z}=\begin{cases}
Z^{-}, & \mathbf{n}^{-}\cdot\mathbf{u}>0\\
Z^{+}, & \text{otherwise}
\end{cases}
\end{equation}
and
\begin{equation}\label{fflux}
    \widehat{f}=\begin{cases}
f^{-}, & \mathbf{n}^{-}\cdot\mathbf{u}<0\ \text{or\ \ensuremath{\partial R}\ensuremath{\subset\partial\Omega}}\\
f^{+}, & \text{otherwise}
\end{cases}
\end{equation}

\begin{defi}[discrete gradient operator]\label{discretegrad}

    We define $\left(\nabla\varphi\right)_{u}$ as the function in $\mathcal{V}$ such that,
    \begin{equation}\label{ugrad}
        -\sum_{R}\left(\left(\nabla\varphi\right)_{u},V\right)_{R}=-\sum_{R}\left(\nabla\varphi,V\right)_{R}+\sum_{R}\langle\varphi^{-}\mathbf{n}^{-},\widehat{V}\rangle_{\partial R/\Gamma},\forall V\in\mathcal{V}
    \end{equation}

    Analogously, define $\left(\nabla\varphi\right)_{d}$ as the function
    in $\mathcal{V}$ such that
    \begin{equation}\label{dgrad}
        \sum_{R}\left(\left(\nabla\varphi\right)_{d},V\right)_{R}=-\sum_{R}\left(\varphi,\nabla\cdot V\right)_{R}+\sum_{R}\langle\widehat{\varphi},V^{-}\cdot\mathbf{n}^{-}\rangle_{\partial R},\forall V\in\mathcal{V}
    \end{equation}
\end{defi}

\begin{prop}\label{u_equiv_d}
    If the flux terms are as defined in Equation(\ref{Zflux}) and Equation(\ref{fflux}), then
    \begin{equation}
        \left(\nabla\varphi\right)_{u}=\left(\nabla\varphi\right)_{d}
    \end{equation}
\end{prop}

\begin{proof}
    Integrate by parts on the right hand side of Equation(\ref{dgrad}),  we have
    \begin{equation*}
        -\sum_{R}\left(\varphi,\nabla\cdot V\right)_{R}+\sum_{R}\langle\widehat{\varphi},V^{-}\cdot\mathbf{n}^{-}\rangle_{\partial R}=\sum_{R}\left(\nabla\varphi,V\right)_{R}+\sum_{R}\langle\widehat{\varphi}-\varphi^{-},V^{-}\cdot\mathbf{n}^{-}\rangle_{\partial R/\Gamma}.
    \end{equation*}

    Summing Equation(\ref{ugrad}) and Equation(\ref{dgrad}) to obtain, 
    \begin{equation}\label{d-u}
        \sum_{R}\left(\left(\nabla\varphi\right)_{d}-\left(\nabla\varphi\right)_{u},V\right)_{R}=\sum_{R}\langle\widehat{\varphi}-\varphi^{-},V^{-}\cdot\mathbf{n}^{-}\rangle_{\partial R/\Gamma}+\sum_{R}\langle\varphi^{-}\mathbf{n}^{-},\widehat{V}\rangle_{\partial R/\Gamma}.
    \end{equation}

    Since
    \begin{equation*}
        \sum_{R}\langle\widehat{\varphi},V^{-}\cdot\mathbf{n}^{-}\rangle_{\partial R/\Gamma}+\sum_{R}\langle\varphi^{-}\mathbf{n}^{-},\widehat{V}\rangle_{\partial R/\Gamma}=\sum_{R}\langle\varphi^{-},V^{-}\cdot\mathbf{n}^{-}\rangle_{\partial R/\Gamma},
    \end{equation*}
    the right hand side of  Equation(\ref{d-u}) must be zero.
    
\end{proof}

As have been discussed in \cite{arnold2002unified}, the LDG weak form can also be written in bilinear form.

\begin{theorem}
    The weak form (\ref{weakform-ldg}) is equivalent to
    \begin{equation*}
        \left(\partial_{t}f,\varphi\right)_{\Omega}+\left(\left(\nabla f\right)_{d},D\cdot\left(\nabla\varphi\right)_{d}\right)_{\Omega}=0
    \end{equation*}
\end{theorem}

\begin{proof}
    By Definition \ref{discretegrad}, the weak form (\ref{weakform-ldg}) is equivalent to
\begin{equation*}
    \begin{split}
        \sum_{R}\left(\partial_{t}f,\varphi\right)_{R}&=-\sum_{R}\left(\left(\nabla\varphi\right)_{u},Z\right)_{R}\\
        \sum_{R}\left(Z,\widetilde{V}\right)_{R}&=\sum_{R}\left(D\cdot\widetilde{Z},\widetilde{V}\right)_{R}\\
        \widetilde{Z}&=\left(\nabla f\right)_{d}
    \end{split}
\end{equation*}

Therefore, the corresponding bilinear form can be derived as follows,
\begin{equation*}
    \begin{split}
        \sum_{R}\left(\partial_{t}f,\varphi\right)_{R}&=-\sum_{R}\left(\left(\nabla\varphi\right)_{u},Z\right)_{R}\\
        &=-\sum_{R}\left(D\cdot\widetilde{Z},\left(\nabla\varphi\right)_{u}\right)_{R}\\
        &=-\sum_{R}\left(\left(\nabla f\right)_{d},D\cdot\left(\nabla\varphi\right)_{u}\right)_{R}\\
        &=-\sum_{R}\left(\left(\nabla f\right)_{d},D\cdot\left(\nabla\varphi\right)_{d}\right)_{R},
    \end{split}
\end{equation*}
where Proposition \ref{u_equiv_d} is used in the last row.
\end{proof}

\subsection{Space discretization}
\subsubsection*{Cut-off domains and boundary conditions}
As have been discussed in \cite{huang2023conservative}, it is assumed that given any $0<\epsilon_p \ll 1$ and $0 < \epsilon_k \ll 1$, there exists bounded domains $\Omega_{px}^{L}=\Omega_{p}^{L}\times\Omega_{x}\subsetneq\mathbb{R}_{p}^{3}\times\Omega_{x}$
and $\Omega_{kx}^{L}\subsetneq\mathbb{R}_{k}^{3}\times\Omega_{x}$ such that  for any $t\geq0$,
\begin{equation*}
    \left|1-\frac{\int_{\Omega_{px}^{L}}f(\mathbf{p},\mathbf{x},t)}{\int_{\mathbb{R}_{p}^{3}\times\Omega_{x}}f(\mathbf{p},\mathbf{x},t)}\right|\leq\epsilon_{p},
\end{equation*}
and
\begin{equation*}
    \left|1-\frac{\int_{\Omega_{kx}^{L}}N(\mathbf{k},\mathbf{x},t)}{\int_{\mathbb{R}_{k}^{3}\times\Omega_{x}}N(\mathbf{k},\mathbf{x},t)}\right|\leq\epsilon_{k}.
\end{equation*}

The cut-off domain $\Omega_{px}^{L}$ for particles is supposed to be adaptive to ensure that $|f|$ and $|\nabla_{\mathbf{p}}f|$ are nearly zero on the boundary, while in our numerical experiments it turns out that, as a result of anisotropic diffusion, there is no need to extend it.

The cut-off domain $\Omega_{kx}^{L}$ for plasmons is supposed to consist of trajectory bundles. We will elaborate on some technical details in Section \ref{sec:implement}.

Under the above assumptions, it is reasonable to solve the equations in cut-off domains. For plasmon \textit{pdf} $N$, there is no need for boundary conditions. For particle \textit{pdf} $f$, we use zero-flux boundary condition in momentum space,
\begin{equation*}
    \left(D[N] \cdot \nabla_{\mathbf{p}}f\right)\cdot\mathbf{n}_{p}=0,\ \forall\mathbf{p}\in\partial\Omega_{p}^{L}.
\end{equation*}

In fact we can also use zero-value boundary condition, the results would not differ much, since the boundary flux terms are below machine epsilon anyway if the cut-off domains are large enough. 

\subsubsection*{Discrete test spaces}
Note that the particle \textit{pdf} is symmetric, i.e. $f(\mathbf{p}, \mathbf{x}, t) = f(p_{\parallel}, p_{\perp}, r, t)$, we have
\begin{equation*}
    \mathbf{p}=(p_{1},p_{2},p_{3})\in\Omega_{p}^{L}\Leftrightarrow(p_{\parallel},p_{\perp})\in\tilde{\Omega}_{p}^{L}\subset\mathbb{R}\times\mathbb{R}^{+}.
\end{equation*}

Let $\{R_{p}\}$ be the rectangular partition of $\tilde{\Omega}^{L}_{p}$, let $\{R_{x}\}$ be the partition of interval $(0,R_{max})$, and let  $\{S\}$ be the coarse-trajectorial partition of $\Omega_{kx}^{L}$.

The test space for particle \textit{pdf} consists of discontinuous piecewise polynomials,
\begin{equation}\label{testspace_f}
    \mathcal{G}_{h}\coloneqq\{f(p_{\parallel},p_{\perp},r):f|_{R_{p}\times R_{x}}\in Q^{\alpha_{1}}(R_{p})\times Q^{\alpha_{2}}(R_{x}),\forall~R_{p},R_{x}\}.
\end{equation}

The test space for plasmon \textit{pdf} consists of indicator functions of trajectory bundles,
\begin{equation}\label{testspace_N}
    \mathcal{N}_{h}\coloneqq\left\{ N(\mathbf{k},\mathbf{x}):N=\sum_{S}n_{S}\mathbf{1}_{S}\right\}.
\end{equation}

\subsection{The conservative scheme}
Before proposing the unconditionally conservative scheme, we introduce two projection operators onto the test spaces defined above. They can be arbitrarily chosen as long as the following conditions are satisfied:

\begin{enumerate}
    \item The projection $\Pi_{px,h}$ onto test space $\mathcal{G}_{h}$ must satisfy that 
    \begin{equation*}
        \ensuremath{\lim_{h\rightarrow0}\Vert\Pi_{px,h}g(\mathbf{p},\mathbf{x})-g(\mathbf{p},\mathbf{x})\Vert_{L^{2}(\Omega_{px}^{L})}=0},\ \forall g\in L^{2}(\Omega_{px}^{L}),
    \end{equation*}
    and
    \begin{equation*}
        \ensuremath{\lim_{h\rightarrow0}\Vert\nabla_{p,h}\left(\Pi_{px,h}E(\mathbf{p})\right)-\nabla_{p}E(\mathbf{p})\Vert_{L^{2}(\Omega_{px}^{L})}=0},
    \end{equation*}
    where the discrete gradient operator $\nabla_{p,h}$ is defined in Definition \ref{discretegrad}.

    \item The projection $\Pi_{kx,h}$ onto test space $\mathcal{N}_{h}$ must satisfy that 
    \begin{equation*}
        \ensuremath{\lim_{h\rightarrow0}\Vert\Pi_{kx,h}\xi(\mathbf{k}, \mathbf{x})-\xi(\mathbf{k}, \mathbf{x})\Vert_{L^{2}(\Omega_{kx}^{L})}=0},\ \forall\xi\in\text{ker}\mathcal{T}\cap L^{2}(\Omega_{kx}^{L}).
    \end{equation*}

\end{enumerate}

There is no need to specify particular projections until we implement them in the numerical examples, our method works with any of them.

Recall the definition of directional differential operator
\begin{equation*}
    \mathcal{L}g\coloneqq k_{\parallel}\frac{\partial E}{\partial p_{\perp}}\frac{p}{p_{\perp}}\frac{\partial g}{\partial p_{\parallel}}+\left(\omega-k_{\parallel}\frac{\partial E}{\partial p_{\parallel}}\right)\frac{p}{p_{\perp}}\frac{\partial g}{\partial p_{\perp}}.
\end{equation*}

We propose a discretized operator as follows,
\begin{equation}\label{L_h}
   \mathcal{L}_{h}g_{h}\coloneqq k_{\parallel,h}\left(\partial_{\perp,h}E_{h}\right)\cdot\frac{p}{p_{\perp}}\left(\partial_{\parallel,h}g_{h}\right)+\left(\omega_{h}-k_{\parallel,h}\left(\partial_{\parallel,h}E_{h}\right)\right)\frac{p}{p_{\perp}}\left(\partial_{\perp,h}g_{h}\right),
\end{equation}
where 
\begin{equation*}
    \begin{split}
        k_{\parallel,h}&\coloneqq\Pi_{kx,h}k_{\parallel},\\
        \omega_{h}&\coloneqq\Pi_{kx,h}\omega,\\
        E_{h}&\coloneqq\Pi_{px,h}E,\\
        \partial_{\parallel,h}g_{h}&\coloneqq\frac{\mathbf{B}}{|\mathbf{B}|}\cdot\nabla_{p,h}g_{h},\\
        \partial_{\perp,h}g_{h}&\coloneqq\left\vert \left(I-\frac{\mathbf{B}}{|\mathbf{B}|}\otimes\frac{\mathbf{B}}{|\mathbf{B}|}\right)\cdot\nabla_{p,h}g_{h}\right\vert.
    \end{split}
\end{equation*}
and the discrete gradient operator $\nabla_{p,h}$ is defined in Definition \ref{discretegrad}.

Based on the above discrete operator $\mathcal{L}_{h}$, we propose an unconditionally conservative semi-discrete weak form.

\begin{theorem}
    If $f_{h}(t)\in\mathcal{G}_{h}$ and $N_{h}(t)\in\mathcal{N}_{h}$ satisfies that
    \begin{equation}\label{semi-discrete}
        \iint_{px}\varphi_{h}\partial_{t}f_{h}+\iint_{kx}\eta_{h}\partial_{t}N_{h}=\iiint_{pkx}N_{h}\mathcal{B}_{\varepsilon}\left(\eta_{h}\mathcal{L}_{h}E_{h}-\omega_{h}\mathcal{L}_{h}\varphi_{h}\right)\mathcal{L}_{h}f_{h},
    \end{equation}
    for any $\varphi_{h}(t)\in\mathcal{G}_{h}$ and $\eta_{h}(t)\in\mathcal{N}_{h}$, then

    \begin{itemize}
    \item Mass Conservation
    \begin{equation*}
        \partial_{t}\mathcal{M}^{h}_{tot}\coloneqq \partial_{t}\left((f_{h},1)_{px}+(N_{h},0)_{kx}\right)=0.
    \end{equation*}
    
    \item Momentum Conservation
    \begin{equation*}
        \partial_{t}\mathcal{P}^{h}_{z,tot} \coloneqq \partial_{t}\left((f_{h},\Pi_{px,h} p_{z})_{px}+(N_{h},\Pi_{kx,h} k_{z})_{kx}\right) = O(e^{-L^{2}}),
    \end{equation*}
    where $L$ is the radius of the cut-off $p$-domain $\tilde{\Omega}_{p}^{L}$.
    \item Energy Conservation
    \begin{equation*}
        \partial_{t}\mathcal{E}^{h}_{tot} \coloneqq \partial_{t}\left((f_{h},\Pi_{px,h}E)_{px}+(N_{h},\Pi_{kx,h}\omega)_{kx}\right)=0.
    \end{equation*}
    
\end{itemize}
\end{theorem}

\begin{proof}
    Mass conservation and energy conservation are trivial by substituting test functions in the weak form (\ref{semi-discrete}). 
    
    For momentum conservation, note that
    \begin{equation*}
        \begin{split}
            &\left(\Pi_{kx,h}k_{z}\right)\mathcal{L}_{h}E_{h}-\omega_{h}\mathcal{L}_{h}\left(\Pi_{px,h}p_{z}\right)\\
            =&\left(\Pi_{kx,h}k_{z}\right)\omega_{h}\frac{p}{p_{\perp}}\left(\partial_{\perp,h}E_{h}\right)-\omega_{h}k_{\parallel,h}\left(\partial_{\perp,h}E_{h}\right)\frac{p}{p_{\perp}}\left(\partial_{\parallel,h}p_{z,h}\right)\\
            =&\omega_{h}k_{\parallel,h}\left(\partial_{\perp,h}E_{h}\right)\frac{p}{p_{\perp}}\left[1-\left(\partial_{\parallel,h}p_{z,h}\right)\right].
        \end{split}
    \end{equation*}

    In general $\left(\frac{\partial p_{z}}{\partial p_{\parallel}}\right)_{h} = 1$, except on some elements near boundary $\partial \Omega^{L}_{p}$. For example, see Equation(\ref{p0grad}). 
    
    Since $\nabla_{p,h}f_{h}$ on the boundary can be arbitrarily small for large enough cut-off domain, we conclude that
    \begin{equation*}
        \iiint_{pkx}N_{h}\mathcal{B}_{\varepsilon}\omega_{h}k_{\parallel,h}\left(\partial_{\perp,h}E_{h}\right)\left[1-\left(\partial_{\parallel,h}p_{z,h}\right)\right]\mathcal{L}_{h}f_{h}=\hat{0}.
    \end{equation*}
\end{proof}

\begin{remark}
    Local discontinuous Galerkin scheme with piecewise constant basis is equivalent to finite difference scheme, and in this particular case, the technique presented here is analogous to the finite difference scheme for Fokker-Planck-Landau equation proposed by Shiroto and Sentoku \cite{shiroto2019structure}.  
\end{remark}

\section{Implementation and complexity} \label{sec:implement}

In this section we elaborate on some technical details during implementation. In addition, the complexity of the proposed algorithms will be discussed.

\subsection{Interpolated Hamiltonian}
Recall the formulation of Poisson bracket
\begin{equation*}
    \mathcal{T}N \coloneqq \left\{ \omega,N\right\} =\frac{\partial N}{\partial r}\frac{\partial\omega}{\partial k_{r}} - \frac{\partial\omega}{\partial r}\frac{\partial N}{\partial k_{r}}.
\end{equation*}

The operator $\mathcal{T}$ determines a flow on $\mathbb{R}^{2}$. However, the Hamiltonian $\omega$ for plasmons is a function of four variables, $\omega = \omega(k_{r}, q_{\phi}, k_{z}, r)$, which means we have to construct the corresponding trajectory bundles for infinitely many $(q_{\phi}, k_{z})$. To avoid that, for implementation, we propose the following interpolation of the dispersion relation.

To begin with, partition the space $\Omega_{a} = \mathbb{R}_{q_{\phi}} \times \mathbb{R}_{k_{z}}$ into rectangular meshes $\{ R \}$, and partition the domain $\Omega_{b} = \mathbb{R}_{k_{r}} \times (0,R_{max})$ into triangular meshes $\{ T \}$.

Next, define the following space of piecewise polynomials.
\begin{equation*}
    \mathcal{H}=\left\{ H(k_{r},r,q_{\phi},k_{z})\in C^{0}(\Omega_{b}):H|_{R\times T}\in Q^{0}(R)\times P^{1}(T),\forall~R \subset \Omega_{a},T \subset \Omega_{b}\right\} .
\end{equation*}

Project the function $\omega = \omega(k_{r}, q_{\phi}, k_{z}, r)$ onto $\mathcal{H}$ to obtain the interpolated Hamiltonian $\overline{\omega} = \overline{\omega}(k_{r}, q_{\phi}, k_{z}, r)$. Obviously, for any given $(q_{\phi}, k_{z})$, $\overline{\omega}$ is a continuous piecewise linear function on $\Omega_{b}$, while for any given $(k_{r}, r)$, $\overline{\omega}$ is a piecewise constant function on $\Omega_{a}$. 

The purpose of such interpolation is as follows.
\begin{theorem}\label{thm_interpolate_traj}
    If $\omega \in \mathcal{H}$, and $S\subset \Omega_{b}$ is a trajectory bundle for $\omega(\cdot, \cdot, q_{\phi}, k_{z})$ with $(q_{\phi}, k_{z}) \in R \subset \Omega_{a}$, then
    \begin{equation*}
        \left\{ \omega,\mathbf{1}_{S\times R}\right\} =0.
    \end{equation*}
    In other words, $S \times R$ is a trajectory bundle w.r.t. the operator $\mathcal{T}N \coloneqq \{ \omega, N\}$.
\end{theorem}

Theorem \ref{thm_interpolate_traj} guarantees that it is sufficient to construct only a finite number of trajectories, as long as there is a representative $(q_{\phi}, k_{z})$ sample for each rectangular mesh $R$.

\subsection{Cut-off domains}
Another concern is the unboundedness of domain $\Omega_{b} = \mathbb{R}_{k_{r}} \times (0,R_{max})$. As introduced in Section \ref{sec:trajbundle}, the connection-proportion algorithm is designed for bounded domains. 

To resolve that, we consider the following cut-off domain, $ \Omega^{L}_{b} = (-L, L) \times (0, R_{max})$, then the connection-proportion algorithm is feasible. Denote the trajectory bundles generated by the algorithm as $\{S^{L}\}$. 

If a trajectory bundle $S^{L}$ does not intersect with the $k_{r}$ boundary $ e^{L} = \{-L,L\} \times (0, R_{max})$, then it is also a trajectory bundle for the original domain $\Omega_{b}$. Putting together all such trajectory bundles,  we have successfully constructed the coarse-trajectorial partition $\{ S \}$ of the cut-off domain $\Omega^{L}_{kx}$ used in Equation(\ref{testspace_N}).

\subsection{Interaction Tensors}
Recall the semi-discrete weak form
\begin{equation*}
    \iint_{px}\varphi_{h}\partial_{t}f_{h}+\iint_{kx}\eta_{h}\partial_{t}N_{h}=\iiint_{pkx}\left(\eta_{h}\mathcal{L}_{h}E_{h}-\omega_{h}\mathcal{L}_{h}\varphi_{h}\right)\mathcal{L}_{h}f_{h}N_{h}\mathcal{B}_{\varepsilon},
\end{equation*}

Suppose that the test space $\mathcal{G}_{h}$ consists of basis functions $\{ \varphi_{i} \}$ and the test space $\mathcal{N}_{h}$ consists of basis functions $\{ \eta_{j} \}$, and let
\begin{equation*}
    \begin{split}
        f_{h}(\mathbf{p},\mathbf{x},t)&=\sum a_{i}(t)\varphi_{i}(\mathbf{p},\mathbf{x}),\\
        N_{h}(\mathbf{k},\mathbf{x},t)&=\sum N_{j}(t)\eta_{j}(\mathbf{k},\mathbf{x}).
    \end{split}
\end{equation*}

The weak form is then equivalent to the following system,
\begin{equation}\label{eq_ode_sys}
    \begin{split}
        \sum_{i}\partial_{t}a_{i}\iint_{px}\phi_{i}\phi_{m}&=-\sum_{n}\sum_{q}\sum_{s}a_{n}N_{q}\omega_{s}\iiint_{pkx}\eta_{q}\eta_{s}\left(\mathcal{L}_{h}\phi_{m}\right)\left(\mathcal{L}_{h}\phi_{n}\right)\mathcal{B}_{\varepsilon},\\
        \sum_{j}\partial_{t}N_{j}\iint_{kx}\eta_{j}\eta_{s}&= \sum_{n}\sum_{q}\sum_{m}a_{n}N_{q}E_{m}\iiint_{pkx}\eta_{q}\eta_{s}\left(\mathcal{L}_{h}\phi_{m}\right)\left(\mathcal{L}_{h}\phi_{n}\right)\mathcal{B}_{\varepsilon}.
    \end{split}
\end{equation}

Denote the mass matrix for particle \textit{pdf} as $M_{im}=\iint_{px}\phi_{i}\phi_{m}$, and denote the mass matrix for plasmon \textit{pdf} as $G_{js}=\iint_{kx}\eta_{j}\eta_{s}$.

Note that the basis functions $\eta_{j}$'s are indicator functions of trajectory bundles, thus the interaction tensor,
\begin{equation*}
    H_{qsmn}=\iiint_{pkx}\eta_{q}\eta_{s}\left(\mathcal{L}_{h}\phi_{m}\right)\left(\mathcal{L}_{h}\phi_{n}\right)\mathcal{B}_{\varepsilon},
\end{equation*}
can be simplified as follows, 
\begin{equation*}
    \begin{split}
        H_{qsmn}
        &=\iiint_{pkx}\eta_{q}\eta_{s}\left(\mathcal{L}_{h}\phi_{m}\right)\left(\mathcal{L}_{h}\phi_{n}\right)\mathcal{B}_{\varepsilon}\\
        &=\iiint_{pkx}\delta_{qs}\eta_{q}\left(\mathcal{L}_{h}\phi_{m}\right)\left(\mathcal{L}_{h}\phi_{n}\right)\mathcal{B}_{\varepsilon}\\
        &=\delta_{qs}\iiint_{pkx}\eta_{q}\left(\mathcal{L}_{h}\phi_{m}\right)\left(\mathcal{L}_{h}\phi_{n}\right)\mathcal{B}_{\varepsilon}\\
        &=\delta_{qs}B_{qmn}.
    \end{split}
\end{equation*}

It is sufficient to compute only the sparse tensor
\begin{equation*}
    B_{qmn}=\iiint_{pkx}\eta_{q}\left(\mathcal{L}_{h}\phi_{m}\right)\left(\mathcal{L}_{h}\phi_{n}\right)\mathcal{B}_{\varepsilon}.
\end{equation*}

And Equation(\ref{eq_ode_sys}) can be written as follows,
\begin{equation}\label{eq_odesys_tensor}
    \begin{split}
        \sum_{i}M_{im}\partial_{t}a_{i}&=-\sum_{n}\sum_{q}\sum_{s}a_{n}N_{q}\omega_{s}\delta_{qs}B_{qmn}=-\sum_{n}\sum_{q}a_{n}N_{q}\omega_{q}B_{qmn},\\
        \sum_{j}G_{js}\partial_{t}N_{j}&=\sum_{n}\sum_{q}\sum_{m}a_{n}N_{q}E_{m}\delta_{qs}B_{qmn}=N_{s}\sum_{n}\sum_{m}a_{n}E_{m}B_{smn}.
    \end{split}
\end{equation}

Recall the definition of test space for particle \textit{pdf},
\begin{equation*}
    \mathcal{G}_{h}\coloneqq\{f(p_{\parallel},p_{\perp},r):f|_{R_{p}\times R_{x}}\in Q^{\alpha_{1}}(R_{p})\times Q^{\alpha_{2}}(R_{x}),\forall~R_{p},R_{x}\}.
\end{equation*}

Here we present the explicit formulation of interaction tensors for the simplest case: when $\alpha_{1} = \alpha_{2} = 0$, i.e. the basis functions are piecewise constant functions on the rectangular meshes. 

It can be verified that the discrete gradient operators read,
\begin{equation}\label{p0grad}
\begin{split}
    \left(\partial_{\parallel,h}g\right)_{i,j}&=\frac{g_{i+1,j}-g_{i,j}}{\left(\Delta p_{\parallel}\right)^{i}},\\
    \left(\partial_{\perp,h}g\right)_{i,j}&=\frac{p_{\perp}^{i,j+1/2}}{p_{\perp}^{i+1/2,j}}\frac{g_{i,j+1}-g_{i,j}}{\left(\Delta p_{\perp}\right)^{j}}.
\end{split}    
\end{equation}

By $i,j$ we mean the $i$th mesh in $p_{\parallel}$ axis and the $j$th mesh in $p_{\perp}$ axis. And for elements on the upper-right boundary we define
\begin{equation*}
    \begin{split}
        g_{I+1,j}=g_{I,j},\\
        g_{i,J+1}=g_{i,J}.
    \end{split}
\end{equation*}

Suppose that
\begin{equation*}
    \begin{split}
        \phi_{m}(p_{\parallel},p_{\perp},r)&=\lambda_{\xi}(r)\tilde{\phi}_{\mu}(p_{\parallel},p_{\perp}),\\
        \phi_{n}(p_{\parallel},p_{\perp},r)&=\lambda_{\zeta}(r)\tilde{\phi}_{\nu}(p_{\parallel},p_{\perp}).
    \end{split}
\end{equation*}

Then by definition
\begin{equation*}
    \begin{split}
        B_{qmn}
        &=\iiint_{pkx}\eta_{q}\left(\mathcal{L}_{h}\phi_{m}\right)\left(\mathcal{L}_{h}\phi_{n}\right)\mathcal{B}_{\varepsilon}\\
        &=\sum_{R_{ij}}\int_{R_{ij}}\left(\iint_{kx}\eta_{q}\lambda_{\xi}\lambda_{\zeta}\left(\mathcal{L}_{h}\tilde{\phi}_{\mu}\right)\left(\mathcal{L}_{h}\tilde{\phi}_{\nu}\right)\mathcal{B}_{\varepsilon}\right)\\
        &\approx\sum_{R_{ij}}\left(\Delta h_{\parallel}\right)^{i}\left(\Delta h_{\perp}\right)^{j}2\pi p_{\perp}^{j}\left.\iint_{kx}\left[\eta_{q}\lambda_{\xi}\lambda_{\zeta}\left(\mathcal{L}_{h}\tilde{\phi}_{\mu}\right)\left(\mathcal{L}_{h}\tilde{\phi}_{\nu}\right)\mathcal{B}_{\varepsilon}\right]\right\vert_{p_{\parallel}^{i},p_{\perp}^{j}}\\
        &=\sum_{R_{ij}}\delta_{\xi\zeta}\left(\Delta h_{\parallel}\right)^{i}\left(\Delta h_{\perp}\right)^{j}2\pi p_{\perp}^{j}\left[\left(\nabla_{h}\tilde{\phi}_{\mu}\right)_{ij}\cdot D^{ij}[\eta_{q},\lambda_{\xi}]\cdot\left(\nabla_{h}\tilde{\phi}_{\nu}\right)_{ij}\right],
    \end{split}
\end{equation*}
where the piecewise constant diffusion coefficient is defined as follows
\begin{equation*}
    D^{ij}[\eta_{q},\lambda_{\xi}]=\left.\left(\iint_{kx}\eta_{q}\lambda_{\xi}\left(\beta_{h}^{ij}\otimes\beta_{h}^{ij}\right)\mathcal{B}_{\varepsilon}\right)\right|_{p_{\parallel}^{i},p_{\perp}^{j}},
\end{equation*}
with
\begin{equation*}
    \beta_{h}^{ij}=\left[\begin{array}{cc}
k_{\parallel,h}\left(\partial_{\perp,h}E_{h}\right)^{ij}\frac{p^{ij}}{p_{\perp}^{j}}, & \left(\omega_{h}-k_{\parallel,h}\left(\partial_{\parallel,h}E_{h}\right)^{ij}\right)\frac{p^{ij}}{p_{\perp}^{j}}\end{array}\right].
\end{equation*}

The advantage of storing $D^{ij}[\eta_{q},\lambda_{\xi}]$ instead of the tensor $B_{qmn}$ is that, in this way, the CFL condition can be explicitly calculated.

\subsection{Complexity analysis}
Suppose that the $r$-domain $(0,R_{max})$ is partitioned into $n_{r}\sim \mathcal{O}(n)$ grids, the cut-off $p$-domain $(P^{l}_{\parallel},P^{r}_{\parallel})\times(0,P^{m}_{\perp})$ is partitioned into $n_{p}\sim \mathcal{O}(n^{2})$ grids. 

In addition, suppose that $n_{z\phi}\sim \mathcal{O}(n^{2})$ is the number of grids in cut-off $(k_{z}, q_{\phi})$-domain $(K^{l}_{z}, K^{r}_{z})\times (0, Q^{m}_{\phi})$, and $n_{\Delta}\sim \mathcal{O}(n^{2})$ is the number of triangular meshes for interpolation of the Hamiltonian $H(r,k_{r})$. For each discrete Hamiltonian, we stratify it into $n_{s} \sim \mathcal{O}(n)$ layers, then $n_{b} \sim \mathcal{O}(n^{3})$ trajectory bundles are generated in total.

\subsubsection*{data}
The degree of freedom for discrete particle \textit{pdf} $f_{h}$ is $n_{r}\times n_{p} \sim \mathcal{O}(n^{3})$, and for discrete plasmon \textit{pdf} $N_{h}$ it is $n_{b} \sim \mathcal{O}(n^{3})$.

\subsubsection*{solving}
The linear map from plasmon \textit{pdf} to element-wise diffusion coefficient is a $\left[n_{r}\times n_{p}\right] \times \left[n_{b} \right]$ sparse matrix with approximately $n_{r}\times n_{p} \times n_{z\phi}\sim \mathcal{O}(n^{5})$ non-zero elements. 

In each time step, the most expensive part is to obtain element-wise diffusion coefficients by matrix-vector product, taking $\mathcal{O}(n^{5})$ flops. Note that since the diffusion process happens only in momentum space, the solving procedure is naturally parallelizable along $r$-axis.

\subsubsection*{preparing}

For each grid in $(k_{z}, q_{\phi})$-domain, we interpolate the Hamiltonian $H(r,k_{r};k^{c}_{z}, q^{c}_{\phi})$ with piecewise linear basis on triangular meshes, which requires evaluating $\omega$ on $n_{\Delta}$ nodes. For each interval $(H_{a}, H_{b})$, constructing the connection matrix as defined in (\ref{connectionmat}) takes $n_{\Delta}$ flops, distinguishing all its connected components takes $\mathcal{O}(n)$ flops. Based on the above analysis, the time complexity for constructing all the trajectory bundles is $\mathcal{O}(n_{z\phi} \times n_{s} \times n_{\Delta}) \sim \mathcal{O}(n^{5})$. However, since the procedure is independent along $k_{z}$-axis and $q_{\phi}$-axis, the time complexity can be reduced in practice through parallel computing. 

For each of these trajectory bundles we have to store its minimal triangle cover, thus the space complexity is $\mathcal{O}(n^{4})$.

Constructing the diffusion coefficients takes $\left[n_{r}\times n_{p}\right] \times \left[n_{z\phi} \times n_{\Delta} \right] \sim \mathcal{O}(n^{6})$ flops, the procedure is also naturally parallelizable since all the evaluations are independent.

\section{Stability and positivity} \label{sec:stability}

In this section, we investigate the stability of the fully discretized nonlinear system. As can be observed, the stability relies on the positivity of plasmon \textit{pdf}, therefore the positivity-preserving  condition will also be discussed.

\begin{lemma}
    Suppose $f_{h}(t)\in\mathcal{G}_{h}$ and $N_{h}(t)\in\mathcal{N}_{h}$ solves Equation(\ref{semi-discrete}) with the following initial conditions,
    \begin{equation*}
        \begin{split}
            f_{h}(\mathbf{p},\mathbf{x},0)&=f_{h}^{0}(\mathbf{p},\mathbf{x}),\\
            N_{h}(\mathbf{k},\mathbf{x},0)&=N_{h}^{0}(\mathbf{k},\mathbf{x}).
        \end{split}
    \end{equation*}

    If the plasmon \textit{pdf} $N_{h}$ is always positive, i.e.
    \begin{equation*}
        N_{h}(\mathbf{k},t)\geq0,~\forall\left(\mathbf{k}, \mathbf{x}, t\right)\in\Omega^{L}_{kx}\times \left[0,+\infty\right),
    \end{equation*}
    then $f_{h}$ has $L^{2}$-stability,
    \begin{equation*}
        \Vert f_{h}\Vert_{L^{2}(\Omega_{px}^{L})}\leq\Vert f_{h}^{0}\Vert_{L^{2}(\Omega_{px}^{L})},
    \end{equation*}
    and $N_{h}$ has bounded $L^{1}_{\omega}$-norm,
    \begin{equation*}
        \Vert N_{h}\Vert_{L_{\omega}^{1}(\Omega_{kx}^{L})}\coloneqq\iint_{kx}|N_{h}|\omega_{h}\leq\mathcal{E}_{tot,h}^{0}+\Vert f_{h}^{0}\Vert_{L^{2}(\Omega_{px}^{L})}\cdot\Vert E_{h}\Vert_{L^{2}(\Omega_{px}^{L})}.
    \end{equation*}
\end{lemma}

There are four types of first order conservative time discretizations,
\begin{equation*}
    \iint_{px}\varphi_{h}\frac{f_{h}^{s+1}-f_{h}^{s}}{\Delta t}+\iint_{kx}\eta_{h}\frac{N_{h}^{s+1}-N_{h}^{s}}{\Delta t}=\begin{cases}
\iiint_{pkx}N_{h}^{s}\mathcal{B}_{\varepsilon}\left(\eta_{h}\mathcal{L}_{h}E_{h}-\omega_{h}\mathcal{L}_{h}\varphi_{h}\right)\mathcal{L}_{h}f_{h}^{s}, & \text{fully-explicit}\\
\iiint_{pkx}N_{h}^{s}\mathcal{B}_{\varepsilon}\left(\eta_{h}\mathcal{L}_{h}E_{h}-\omega_{h}\mathcal{L}_{h}\varphi_{h}\right)\mathcal{L}_{h}f_{h}^{s+1}, & \text{semi-implicit}\\
\iiint_{pkx}N_{h}^{s+1}\mathcal{B}_{\varepsilon}\left(\eta_{h}\mathcal{L}_{h}E_{h}-\omega_{h}\mathcal{L}_{h}\varphi_{h}\right)\mathcal{L}_{h}f_{h}^{s}, & \text{semi-implicit}\\
\iiint_{pkx}N_{h}^{s+1}\mathcal{B}_{\varepsilon}\left(\eta_{h}\mathcal{L}_{h}E_{h}-\omega_{h}\mathcal{L}_{h}\varphi_{h}\right)\mathcal{L}_{h}f_{h}^{s+1}, & \text{fully-implicit}
\end{cases}
\end{equation*}

The first and third row are explicit for particle \textit{pdf} $f_{h}(t)$, therefore the time stepsize $\Delta t$ has to satisfy the CFL condition for diffusion equations in order to preserve the $L^{2}$-stability. The second and the fourth row are implicit for particle \textit{pdf} $f_{h}(t)$, therefore unconditionally $L^{2}$-stable. Nevertheless, all of them relies on the positivity of plasmon \textit{pdf} $N_{h}$. And in what follows we will prove that $N^{s}_{h}$ remains positive as long as the time stepsize $\Delta t$ is small enough. For details, we refer the readers to \cite{huang2023conservative}.

\begin{theorem}
    Let $f^{\nu}_{h}$ and $N^{\nu}_{h}$ be the solutions of the fully discrete equation, with explicit Euler scheme.

    If $\lVert f_{h}^{\nu}\rVert_{L^{2}(\Omega_{px}^{L})}\leq\lVert f_{h}^{0}\rVert_{L^{2}(\Omega_{px}^{L})}$ for any $\nu$, then there exists
    \begin{equation*}
        \Delta t_{M}\coloneqq C\frac{1-\delta}{\lVert f_{h}^{0}\rVert_{L^{2}(\Omega_{px}^{L})}},
    \end{equation*}
    such that
    \begin{equation*}
        \iint_{R_{kx}}N_{h}^{\nu+1}\geq\delta \cdot \iint_{R_{kx}}N_{h}^{\nu}.
    \end{equation*}
    for any trajectory bundle $R_{kx}$, as long as $\Delta t < \Delta t_{M}$.
\end{theorem}
\begin{proof}
    Recall the second row of ODE system (\ref{eq_odesys_tensor}),
    \begin{equation*}
        \sum_{j}G_{js}\partial_{t}N_{j}=N_{s}\sum_{n}\sum_{m}a_{n}E_{m}B_{smn}.
    \end{equation*}
    By definition the mass matrix $G_{js}$ is diagonal, hence $G_{js}= \delta_{js} G_{ss}$, and it follows that
    \begin{equation*}
        G_{ss}\partial_{t}N_{s}=N_{s}\sum_{n}\sum_{m}a_{n}E_{m}B_{smn}.
    \end{equation*}
    With explicit Euler scheme, we have
    \begin{equation*}
        N_{s}^{\nu+1}=N_{s}^{\nu}\left(1+\frac{\Delta t}{G_{ss}}\cdot\sum_{n}\sum_{m}a^{\nu}_{n}E_{m}B_{smn}\right).
    \end{equation*}
    By the $L^2$-stability of particle \textit{pdf} $f^{\nu}_{h}$ we have
    \begin{equation*}
        \left\vert\frac{\Delta t}{G_{ss}}\cdot\sum_{n}\sum_{m}a_{n}^{\nu}E_{m}B_{smn}\right\vert\leq C_{1}\lVert a_{n}^{\nu}\rVert_{l^{2}}\Delta t\leq C_{2}\lVert f_{h}^{0}\rVert_{L^{2}(\Omega_{px}^{L})}\Delta t
    \end{equation*}
    Define the upper bound
    \begin{equation*}
        \Delta t_{M}=\frac{1}{C_{2}}\frac{1-\delta}{\lVert f_{h}^{0}\rVert_{L^{2}(\Omega_{px}^{L})}},
    \end{equation*}
    then
    \begin{equation*}
        N_{s}^{\nu+1} > \delta \cdot N_{s}^{\nu}
    \end{equation*}
    as long as $\Delta t < \Delta t_{M}$.
\end{proof}

\section{Numerical results}\label{sec:cylinder_results}
\subsection{Problem setting}
Analogous to \cite{huang2023conservative}, only anomalous Doppler resonance with $l=1$ is considered. For dispersion relation $\omega(\mathbf{k},\mathbf{x})$ we take the lowest branch. 

Set the cut-off computational domains as follows:
\begin{equation*}
    \begin{split}
        \Omega_{px}^{L}&=\left\{ \left(r,p_{\parallel},p_{\perp}\right):r\in\left(0,R_{max}\right),\ p_{\parallel}\in\left(10m_{e}c,25m_{e}c\right),\ p_{\perp}\in\left(0,15m_{e}c\right)\right\} \\
        \Omega_{kx}^{L}&=\left\{ \left(r,k_{r},q_{\phi},k_{z}\right):r\in\left(0,R_{max}\right),\ k_{r}\in\left(-\frac{\omega_{0}}{c},\frac{\omega_{0}}{c}\right),\ q_{\phi}\in\left(0,0.5 \frac{\omega_{0}R_{max}}{c}\right),\ k_{z}\in\left(0,\frac{\omega_{0}}{c}\right)\right\} 
    \end{split}
\end{equation*}

The domain $\Omega^{L}_{px}$ is partitioned into rectangular meshes, with $75$ meshes in $p_{\parallel}$ axis, $75$ meshes in $p_{\perp}$ axis,  and $20$ meshes in $r$ axis.

The domain $\Omega^{L}_{kx}$ is partitioned as follows. First of all, partition the $\phi z$-domain into $20\times40$ rectangular meshes. In each $\phi z$-mesh interpolate $\omega(\cdot, \cdot, q_{\phi}, k_{z})$ on $20\times20\times2$ triangular meshes, as illustrated in Figure \ref{fig:traj}. For each Hamiltonian $\omega(\cdot, \cdot, q_{\phi}, k_{z})$, we construct trajectory bundles generated by $10$ boxes with the connection-proportion algorithm.

Choose piecewise constant functions as the test/trial spaces:
\begin{equation*}
    \begin{split}
        \mathcal{G}_{h}&=\left\{ \varphi\left(r,p_{\parallel},p_{\perp}\right):\varphi|_{R_{p}\times R_{x}}\in Q^{0}(R_{p})\times Q^{0}(R_{x}),\forall~R_{p},R_{x}\right\} ,\\
        \mathcal{W}_{h}&=\left\{ \eta\left(r,k_{r},q_{\phi},k_{z}\right):\eta=\sum_{\widetilde{S}}n_{\widetilde{S}}\mathbf{1}_{\widetilde{S}}\right\} .
    \end{split}
\end{equation*}

We choose the following projection operator $\Pi_{kx,h}$:
\begin{equation*}
    \Pi_{kx,h}G(\mathbf{k},\mathbf{x})=\sum_{\widetilde{S}}\left(\sup_{(\mathbf{k},\mathbf{x})\in\widetilde{S}}G(\mathbf{k},\mathbf{x})\right)\cdot1_{\widetilde{S}}.
\end{equation*}

Meanwhile, we define projection $\Pi_{p,h}$ as 
\begin{equation*}
    \Pi_{p,h}G(p_{\parallel},p_{\perp})\coloneqq G\left(p_{\parallel}^{i-1/2},p_{\perp}^{j-1/2}\right),\ \forall(p_{\parallel},p_{\perp})\in R_{p}^{i,j}.
\end{equation*}

Consider a magnetized plasma with non-uniform electron density:
\begin{equation*}
    n_{e}(r)=n_{0}\left[1-\left(\frac{r}{R_{max}}\right)^{2}\right],
\end{equation*}
embedded in a uniform external magnetic field
\begin{equation*}
    \mathbf{B}(r)=B_{0}\mathbf{e}_{z}.
\end{equation*}

Analogous to \cite{huang2023conservative}, we only compute the ``bump" part of electron \textit{pdf} $f(\mathbf{p},\mathbf{x},t)$, which takes the following initial configuration:
\begin{equation*}
    \left.f(p_{\parallel},p_{\perp},r)\right\vert _{t=0}=\left[10^{-5}\frac{1}{\sqrt{\pi}}\exp\left(-\left(\frac{p_{\parallel}}{m_{e}c}-20\right)^{2}-\left(\frac{p_{\perp}}{m_{e}c}\right)^{2}\right)\right]\frac{n_{0}}{m_{e}^{3}c^{3}}\left(\frac{1}{R_{max}}\right)^{3}.
\end{equation*}

Meanwhile, the plasmon \textit{pdf} $N(\mathbf{k},\mathbf{x},t)$ is initialized as follows:
\begin{equation*}
    \left.N\left(r,k_{r},q_{\phi},k_{z}\right)\right\vert _{t=0}=10^{-5}\frac{n_{0}mc^{2}}{(\omega_{0}/c)^{3}}\frac{1}{\hbar\omega_{0}}\left(\frac{1}{R_{max}}\right)^{3}.
\end{equation*}


\subsection{Trajectory bundles}
As shown in Figure \ref{fig:traj}, the triangular partition of the $(r,k_{r})$-domain is done by dividing every rectangular mesh into two. The connection-proportion algorithm successfully distinguishes different trajectory bundles in the inverse image of the same box(in this case, interval).

\begin{figure}
    \centering
    \begin{subfigure}[b]{0.45\textwidth}
        \centering
        \includegraphics[width=\textwidth]{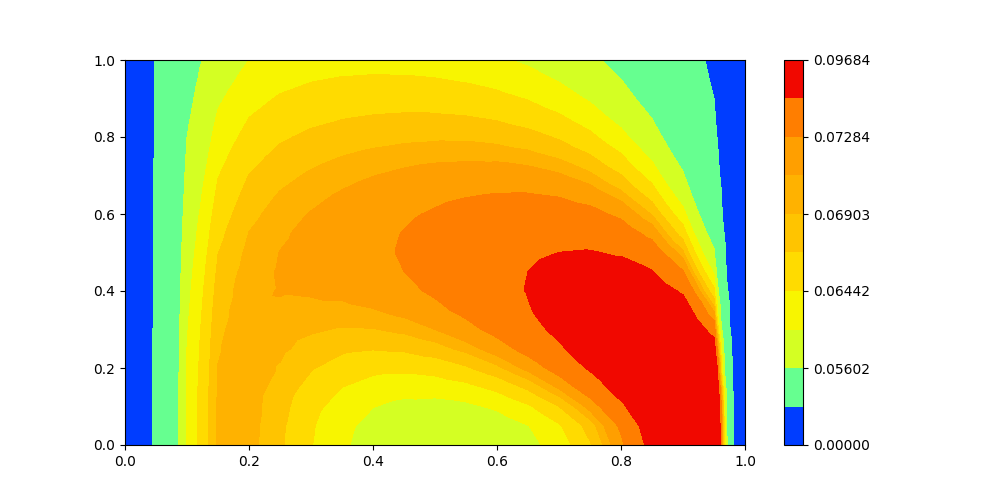}
        \label{fig_Hamiltonian}
        \caption{$\omega(r,k_{r};q_{\phi}, k_{z})/\omega_{0}$ at $(q_{\phi}, k_{z}) = (\frac{9\omega_{0}}{80c}, \frac{9\omega_{0}R_{max}}{80c})$}
    \end{subfigure}
    \hfill
    \begin{subfigure}[b]{0.45\textwidth}
        \centering
        \includegraphics[width=\textwidth]{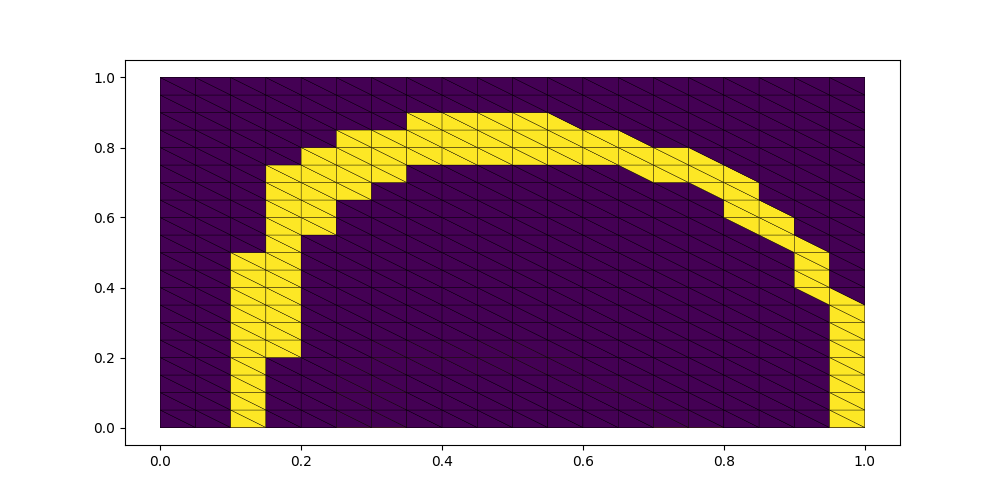}
        \caption{The minimal triangle cover of $S_{11}$}
    \end{subfigure}
    \vskip\baselineskip
    \begin{subfigure}[b]{0.45\textwidth}
        \centering
        \includegraphics[width=\textwidth]{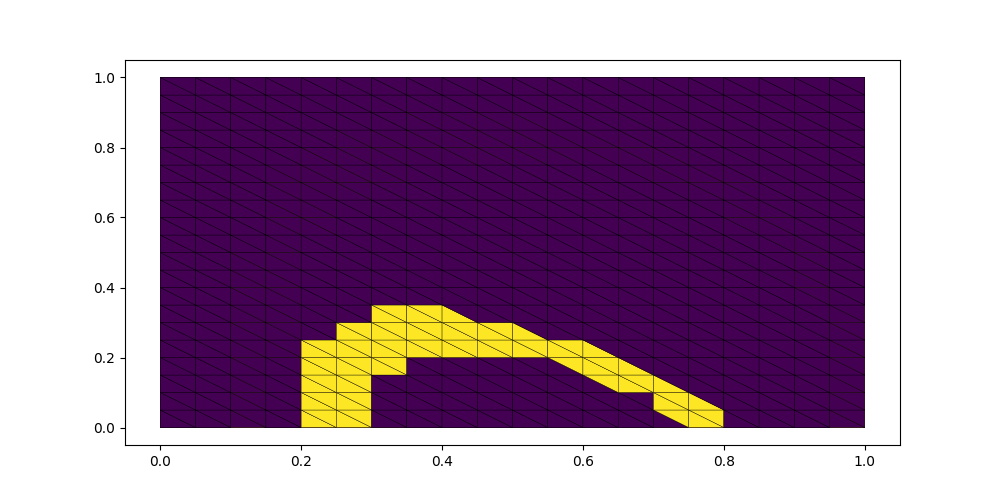}
        \caption{The minimal triangle cover of $S_{12}$}
    \end{subfigure}
    \hfill
    \begin{subfigure}[b]{0.45\textwidth}
        \centering
        \includegraphics[width=\textwidth]{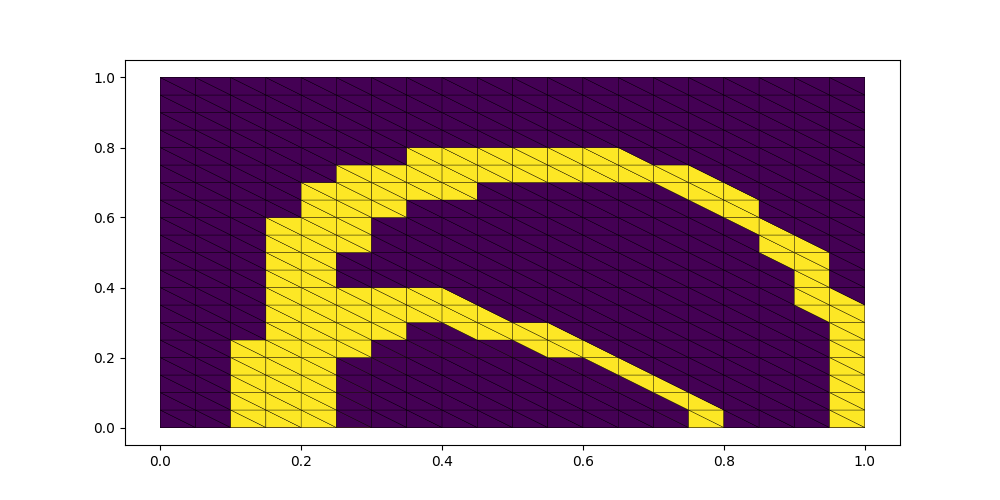}
        \caption{The minimal triangle cover of $S_{13}$}
    \end{subfigure}
    \caption[Trajectory bundles and their minimal triangle covers]{Trajectory bundles and their minimal triangle covers. The $x$-axis represents $r/R_{max}$, and the $y$-axis represents $k_{r}c/\omega_{0}$.  Since the Hamiltonian is symmetric for $\pm k_{r}$, we only plot half of the domain. Note that $S_{11}$ and $S_{12}$ are two trajectory bundles generated by the same Hamiltonian range interval. And $S_{13}$ is not a single strip because it contains a saddle point.}
    \label{fig:traj}
\end{figure}

\subsection{Temporal evolution and spatial inhomogeneity}
In Figure \ref{fig:cylind_evol} we show the evolution of electron \textit{pdf} $f(p_{\parallel},p_{\perp},r, t)$ at $r=\frac{31R_{max}}{40}$. In Figure \ref{fig:cylind_spatial_f} the electron \textit{pdf} at the same time point $t=3\times 10^{6} \frac{1}{2\pi \omega_{0}}$ in different positions is presented. 

\begin{figure}[htbp!]
    \centering
    \begin{subfigure}[b]{\textwidth}
        \centering
        \includegraphics[width=0.8\textwidth]{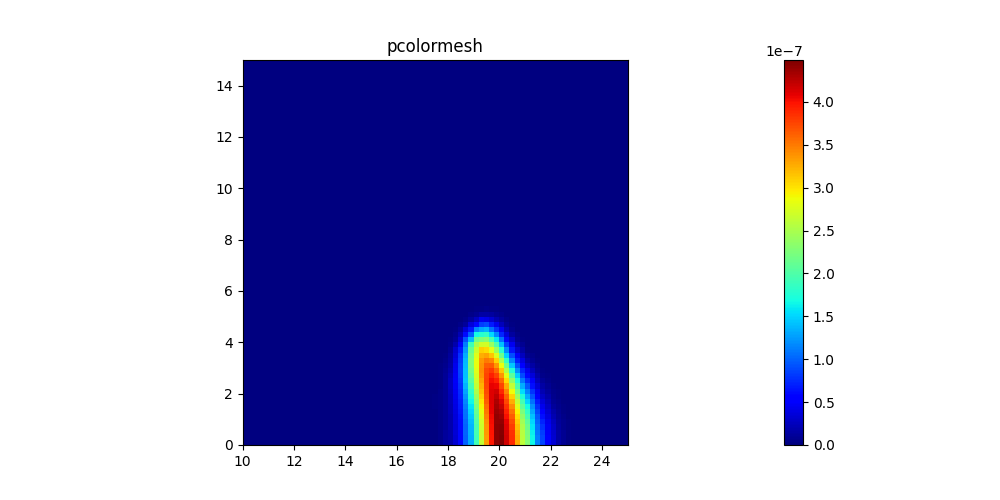}
    \end{subfigure}
    \begin{subfigure}[b]{\textwidth}
        \centering
        \includegraphics[width=0.8\textwidth]{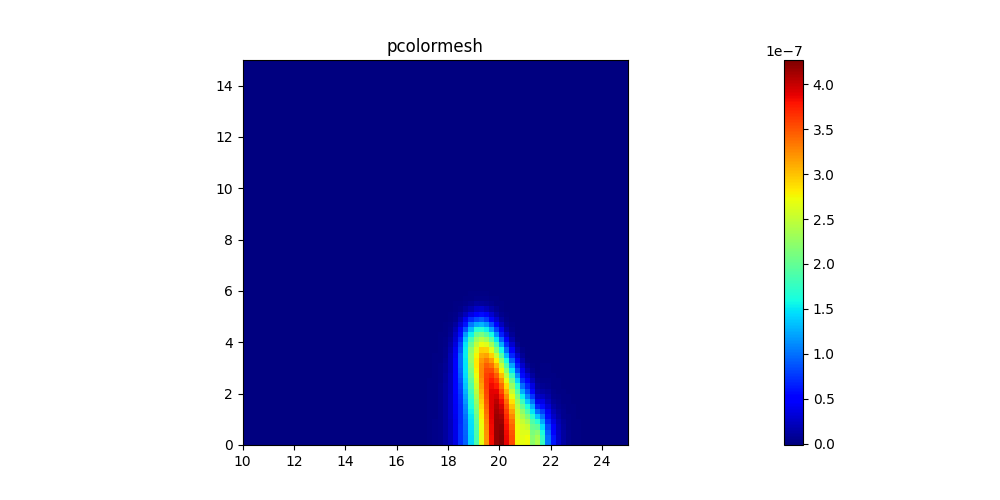}
    \end{subfigure}
    \begin{subfigure}[b]{\textwidth}
        \centering
        \includegraphics[width=0.8\textwidth]{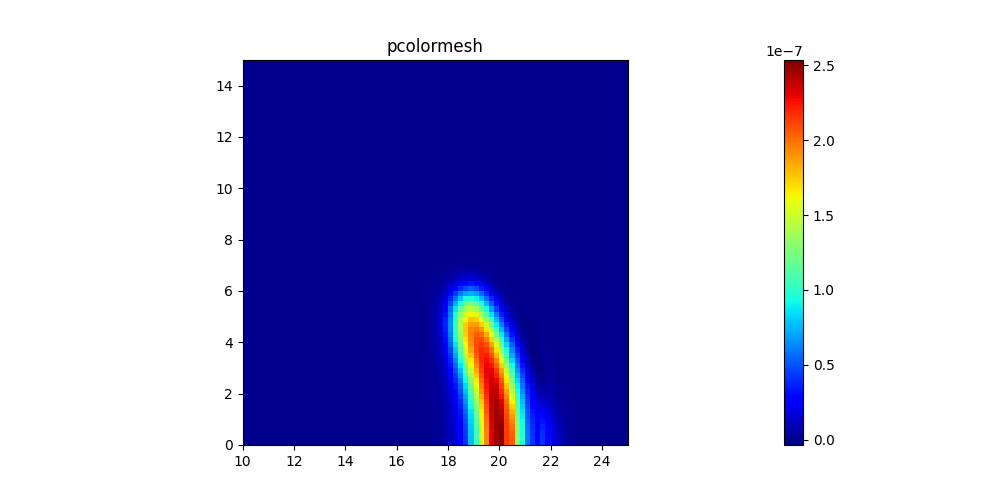}
    \end{subfigure}
    \caption{Spatial inhomogeneity}
    \label{fig:cylind_spatial_f}
\end{figure}

\begin{figure}[htbp!]
    \centering
    \begin{subfigure}[b]{\textwidth}
        \centering
        \includegraphics[width=0.8\textwidth]{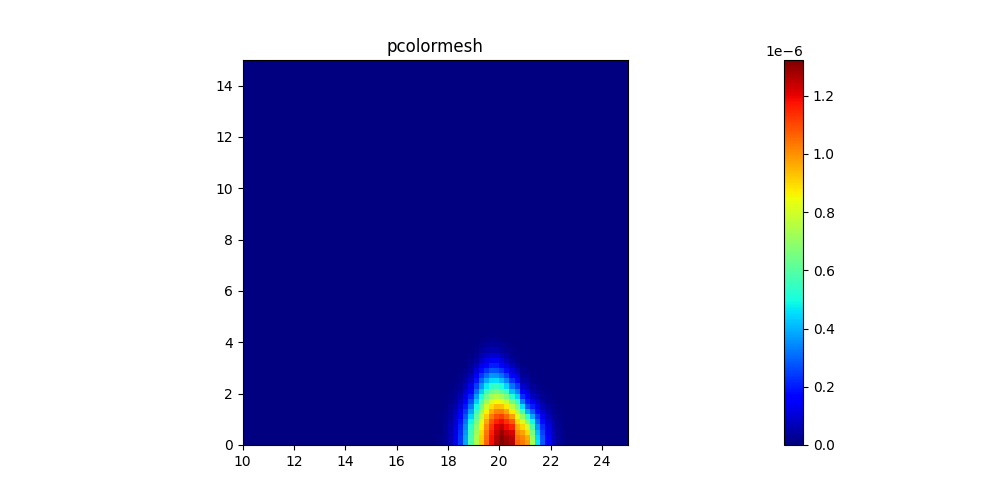}
    \end{subfigure}
    \begin{subfigure}[b]{\textwidth}
        \centering
        \includegraphics[width=0.8\textwidth]{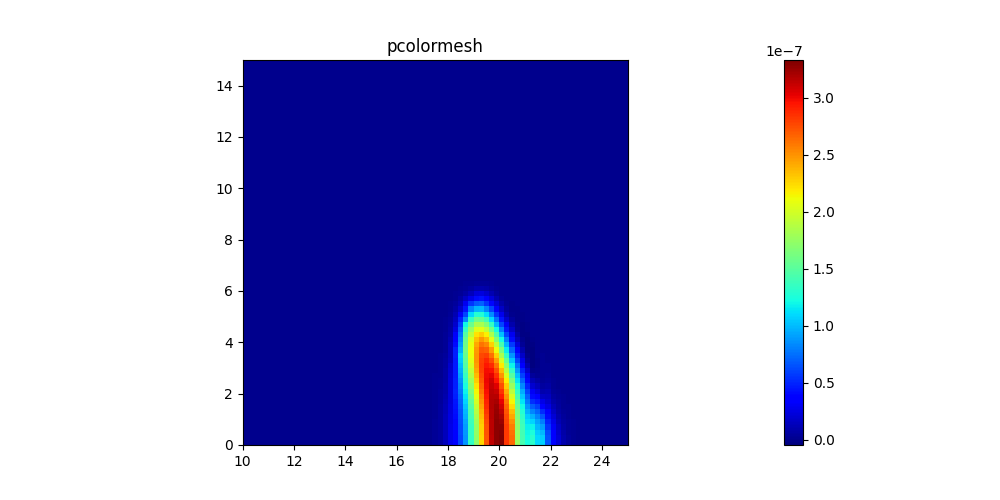}
    \end{subfigure}
    \begin{subfigure}[b]{\textwidth}
        \centering
        \includegraphics[width=0.8\textwidth]{Fig/density/r15time3e6.png}
    \end{subfigure}
    \caption{Temporal evolution.}
    \label{fig:cylind_evol}
\end{figure}

\subsection{Conservation verification}
For the evolution of the electron-plasmon system momentum and energy, see Figure(\ref{fig:conservation}).

With $T_{max}=3.86\times10^{6}\frac{1}{2\pi\omega_{0}}$, we have the following relative errors:
\begin{equation*}
    \begin{split}
        e_{rel}(\mathcal{M}_{tot,h})&=4.6\times 10^{-16},\\
        e_{rel}(\mathcal{P}_{\parallel,tot,h})&=1.8\times 10^{-14},\\
        e_{rel}(\mathcal{E}_{tot,h})&= 5.5 \times 10^{-16}.
    \end{split}
\end{equation*}

\begin{figure}[htbp!]
    \centering
    \begin{subfigure}[b]{0.4\textwidth}
        \centering
        \includegraphics[width=\textwidth]{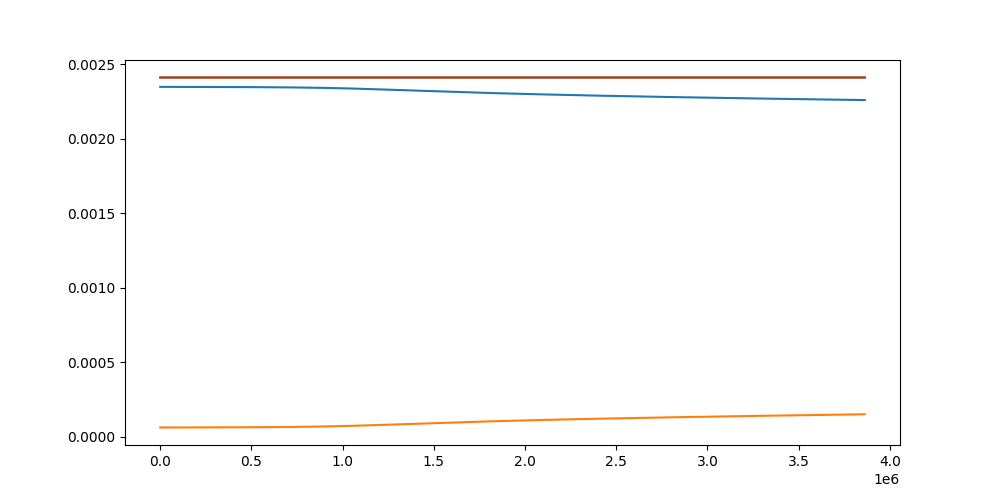}
        \caption{Momentum evolution.}
    \end{subfigure}
    \begin{subfigure}[b]{0.4\textwidth}
        \centering
        \includegraphics[width=\textwidth]{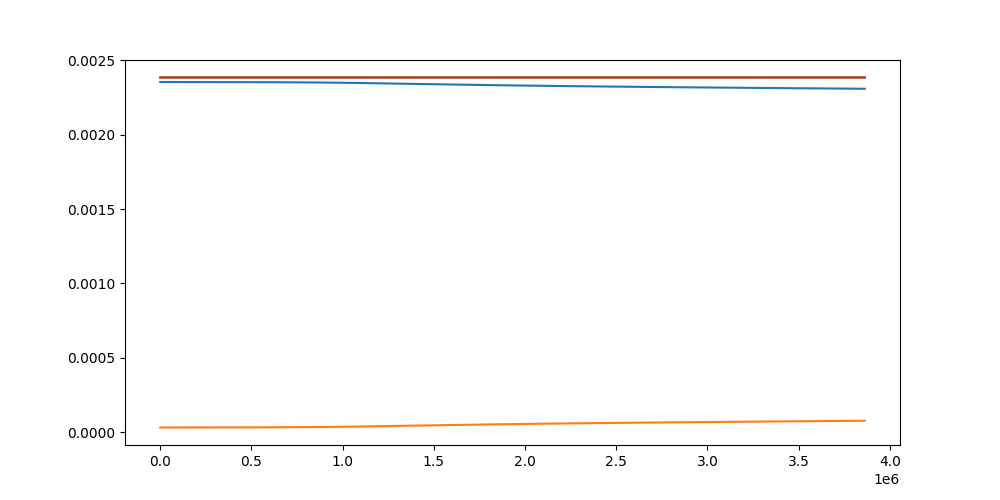}
        \caption{Energy evolution.}
    \end{subfigure}
    \caption{Conservation laws}
    \label{fig:conservation}
\end{figure}

\section{Summary}\label{sec:summary}
In this paper, we introduced a structure-preserving solver for particle-wave interaction in a cylinder with radially non-uniform plasma density. We preserve all the conservation laws during ``collision" by adopting the unconditionally conservative weak form. For the fast periodic advection of plasmons in the phase space, we proposed a novel averaging method, which significantly reduces the computational cost without violating the Hamiltonian structure. All these properties have been verified by numerical experiments, where we observe different diffusion rate in different positions.

In the future, we plan to extend the solver for more sophisticated models, enabling the full-scale simulation of runaway electrons.

\section*{Acknowledgment}
The authors thank and gratefully
acknowledge the support from the Oden Institute of Computational
Engineering and Sciences and the Institute for Fusion Studies at the University of Texas at Austin. This project 
was supported by funding from NSF DMS: 2009736, NSF grant
DMS-2309249 and DOE DE-SC0016283
project Simulation Center for Runaway Electron Avoidance and Mitigation.

\bibliographystyle{plain}
\bibliography{main.bib}
\end{document}